\documentclass[a4j,12pt]{article}
\usepackage{amsmath}
\usepackage{amsthm}
\usepackage{amssymb}
\usepackage{amscd}
\usepackage[dvips]{graphicx}
\newtheorem{definition}{Definition}[section]
\newtheorem{theorem}[definition]{Theorem}
\newtheorem{proposition}[definition]{Proposition}
\newtheorem{lemma}[definition]{Lemma}
\newtheorem{cor}[definition]{Corollary}
\newtheorem{example}[definition]{Example}
\newtheorem{remark}[definition]{Remark}
\newtheorem{conjecture}[definition]{Conjecture}

\newtheorem{question}[definition]{Question}

\newtheorem{sublemma}[definition]{Sublemma}
\newcommand{\s}{^{\sharp}}
\newcommand{\n}{^{\natural}}

\newcommand{\gH}{\mathfrak{H}}
\newcommand{\gI}{\mathfrak{I}}

\newcommand{\infgal}{\mathrm{Inf}\text{-}\mathrm{gal}\, }
\newcommand{\iqd}{iterative $q$-difference }
\newcommand{\qsi}{$q$-SI $\sigma$-differential }

\newcommand{\com}{\mathbb {C}} 

\newcommand{\gd}{\delta}
\newcommand{\D}{\mathcal{D}}

\newcommand{\N}{\mathbb{N}}
\newcommand{\G}{\mathbb{G}}

\newcommand{\calH}{\mathcal{H} \,}
\newcommand{\Hom}{\mbox{ $\mathrm{Hom}$ }}
\newcommand{\Id}{\mathrm{Id}}

\newcommand{\M}{\mathbf{M}} 
\newcommand{\Q}{\mathbb{Q}}

\newcommand{\Z}{\mathbb{Z}}
\newcommand{\K}{\mathcal{K}}
\newcommand{\eL}{\mathcal{L}}

\newcommand{\mapue}[1]{%
     \smash{\mathop{%
      \hbox to 1cm{\rightarrowfill}}\limits^{#1}}}
\newcommand{\NCF}{\mathcal{NCF}\sb{L/k}}
\newcommand{\NCA}{\text{$(NCAlg/L^{\natural})$}}
\begin{document}
\title{Quantization of Galois theory, \\
Examples and observations}
\author{Katsunori Saito and 
Hiroshi Umemura \\  
Graduate School of Mathematics \\
Nagoya University\\ \ 
\\
Email\quad {\small 
m07026e@math.nagoya-u.ac.jp and 
 umemura@math.nagoya-u.ac.jp} }
\maketitle
\begin{abstract}
Heiderich \cite{hei10} discovered that we can apply the Hopf Galois theory also to non-linear equations. We showed in \cite{ume11}
that so far as we consider linear difference -differential  equations, Galois group is a linear algebraic group. We give three examples of
non-linear difference-differential equations 
 in which quantum groups naturally arise
 as Galois groups.  
\end{abstract}
\section{Introduction} 
The pursuit of $q$-analogue of hypergeometric functions goes back to 
the 19th century. Galois group of a $q$-hypergeometric function is not a quantum 
group but it is a linear algebraic group. 
This shows that  we consider a $q$-deformations of the hypergeometric equation, Galois theory is not quantised.  
In fact, 
generally we know that  Galois group of a linear 
difference equation is a linear algebraic group.  
Y.~Andr\'e \cite{and01} wis the first who studied linear difference-differential equations in the 
framework of non-commutative geometry. He encountered only linear algebraic groups treating linear difference-differential equations. See also Hardouin \cite{har10}. 
We clarified the situation in \cite{ume11}. 
So far as we study 
linear difference-differential equations, how twisted or non-commutative the ring of difference and differential operators are, Galois group according to general Hopf Galois theory is a linear algebraic group.   
\par 
  So it is natural to wonder  how about considering non-linear difference-differential equations. We proposed to study the $q$-Painlev\'e equations in \cite{ume11}. 
  We answer this question in the following way. 
  We see quantization of Galois graroup for much simpler equations than the 
  $q$-Painlev\'e equations (Sections \ref{10.4a}, \ref{10.4b} and \ref{10.4c}). 
  \par
Except for  Lie algebras, all the rings or algebras are associative $\Q$-algebras and 
contian the unit element.  Morphisms between 
them are are unitary. For a commutative algebra
 $A$, we denote by $(Alg/ A)$ the category of $A$-algebras, which we sometimes denote by $(CAlg/ A)$ 
to emphasize that we are dealing with commutative $A$-algebras. In fact, to study quantum groups, we have to also consider non-commutative $A$-algebras.  We denote by $(NCAlg/A)$ the category of not necessarily commutative $A$-algebras $B$ such that 
$A$ ( or to be more logic, the image of $A$ in
 $B$) is contained in the center of $B$. 
\section{Foundation of a general Galois theory \cite{ume96.2}, \cite{ume06}, \cite{ume07}}
\subsection{Notation}
\par  
Let us recall basic  notation. 
Let $(R, \gd )$ be a difference ring so that $\gd:R \to R$ is a derivation of a
commutative ring $R$ of characteristic $0$. When there is no danger of confusion of the derivation 
 $\gd$, we simply say the differential ring $R$ without referring to the derivation 
  $\gd$. We often have to talk, however, about the abstract ring $R$ that we denote by $R\n$. 
  For a commutative ring $S$ of characteristic $0$, the power series ring 
  $S[[X]]$ with derivation $d/dX$ gives us an example of differential ring.  
\subsection{General Galois theory of differential field extenwions}
  Let us start by recalling our general Galois theory of differential field extensions.
 \subsubsection{Universal Taylor  morphism}
 \par 
 Let $(R, \,  \gd )$ be a differential ring and $S$ a commutative ring. 
 A Taylor morphism is a differential morphism 
 \begin{equation}\label{a2.1}   
 (  R, \,\gd ) \to (S[[X]], \, d/dX   ).
 \end{equation}
 Given a differential ring $(R, \, \gd )$, 
 among the Taylor morphisms \eqref{a2.1}, there exists the universal one. 
 In fact,   for an element $a\in R$, we define the power series 
  $$
  \iota (a) = \sum \sb{n=0} ^{\infty} \frac{1}{n!}\gd^{n}(a)X^n \in  R\n [[X]].
  $$ 
Then the map
\begin{equation}\label{9.19d}
\iota \colon (R, \, \gd ) \to (R\n [[X]],\,  d/dX )
\end{equation}
is the universal Taylor morphism.
\subsubsection{Galois hull $\eL /\K$ for a differential field extension $L/k$}\label{9.19h}
Let $(L, \, \gd)/(k, \, \gd  )$ be a differential field extension such that 
the abstract field $L\n $ is finitely
 generated over the abstract base field $k\n$. We constructed the Galois hull 
 $\eL / \K$ in the following manner.  

 We take a mutually commutative basis 
 $$
 \{ D \sb 1, \,  D \sb 2 , \cdots 
 , D\sb d \}
 $$
  of the $L\n$-vector space $\mathrm{Der} \, (L\n / k\n )$ 
 of $k\n$-derivations of the abstract field $L\n$. So we have 
 $$
[ D\sb i , D\sb j ] =D\sb i D\sb j - D\sb j D\sb i=  0 \qquad \text{for }  1 \le i, \, j \le d. 
 $$
\par 
Now we introduce a partial 
differential structure on the abstract field $L\n$ 
using the derivations $\{D\sb 1 , \, D\sb 2, \,
 \cdots , D\sb d   \}$. 
  Namely we set 
$$
L\s := ( L\n , \,\{ D\sb 1 , \, D\sb 2, \,  \cdots , D\sb d    \} )
$$  
that is a partial differential  field. 
Similarly we define a differential structure on the  power series ring 
$L\n [[X]]$ with coefficients  in $L\n$ by  considering     
the derivations $$\{D\sb 1 , \, D\sb 2, \,
 \cdots , D\sb d   \}$$ 
 that operate on the coefficients of the power series. 
  In other words, 
 we work with the differential ring 
 $L\s [[X]] $.  
 So the power series ring 
 $L\s[[X]]$ has  differential structure 
 defined by the differentiation  $d/dX$ with respect to the variable $X$ 
    and the set 
 $$\{ D\sb 1 , \, D\sb 2, \,  \cdots , D\sb d    \}$$ 
 of derivations. Since there is no danger of confusion of the choice of the differential operator $d/dX$, 
 we denote this differential ring by 
 $$
  L\s [[X]].    . 
 $$
We have the universal Taylor morphism 
\begin{equation}\label{m27.0.5}
\iota \colon L \to L\n [[X]]
\end{equation}
that is a differential morphism. We added further the
$\{ D\sb 1 , \, D\sb 2, \,  \cdots , D\sb d    \}$-differential structure 
on $L\n [[X]] $ or we replace the target space 
$ L\n [[X]] $ of the universal Taylor morphism \eqref{m27.0.5}
 by $L\s [[X]] $ so that we have 
 $$
 \iota \colon L \to  L\s [[X]] . 
 $$
\par 
In Definition \ref{a4.6} below, we work in the differential ring 
$L\s[[X]] $ with differential operators 
$d/dX $ and   
$
\{ D\sb 1 , \, D\sb 2, \,  \cdots , D\sb d    \}.
$
We identify  the differential field  $L\s$ 
with 
the set of power series
consisting only of constant terms.    
Namely, 
$$
L \s = \{ \sum \sb {n=0}^{\infty} a\sb n X^n \in  L\s [[X]]  \, | \, \,\text{\it \/The 
coefficients }
a\sb n =0 \text{ for every } n\ge 1 \}.
$$
Therefore $L\s$ is a differential sub-field of the differential 
ring $L\s [[X]] $. The differential operator $d/dX$ 
kills $L\s$.  
  Similarly, 
we set  
$$
k \s := \{
 \sum \sb {n=0}^{\infty} a\sb n X^n \in  L\s [[X]]  \, | \, \,\text{\it \/The 
coefficients $a\sb 0 \in k$ and }
a\sb n =0 \,\text{\it \/ for every } n\ge 1 \}.
$$
So all the differential operators 
$d/dX, \, D\sb 1, \, D\sb 2, \cdots , D\sb d$ 
 act trivially on 
$k\s$ and so $k\s$ is a differential sub-field of $L\s$ and hence 
of the differential algebra $L\s [[X]] $. 
\begin{definition}\label{a4.6}
The Galois hull $\eL / \K $ is the differential sub-algebra 
of $L\s[[X]] $, where 
$\eL$ is the differential sub-algebra generated by the 
image $\iota (L)$ and $L\s$ 
and $\K $ is the sub-algebra generated by the image $\iota (k ) $ 
and $ L\s$. So $\eL / \K$ is a differential algebra  extension with differential operators $d/dX $ and  
$
\{ D\sb 1 , \, D\sb 2, \,  \cdots , D\sb d    \}.
$
\end{definition} 
\subsubsection{Universal Taylor morphism for a partial differential ring}
\par   
The universal Taylor morphism has a generalization for 
partial differential ring. 
Let $$(R, \,\{ \partial \sb 1 ,\, \partial \sb 2 ,\, \cdots , \partial \sb d \}) $$ be a partial differential ring. So $R$ is a commutative ring of characteristic $0$ and 
$\partial \sb i \colon  R \to R$ are mutually commutative derivations. 
For a  ring $S$, the power series ring 
$$
( S[[ X\sb 1,\,  X\sb 2,\,  \cdots , X\sb d ]], \,\{ 
\frac{\partial }{\partial X\sb 1 }, \,
\frac{\partial }{\partial X\sb 2 }, \, \cdots ,  
  \frac{\partial }{\partial X\sb d } \})  
$$
gives us an example of 
partial differential ring. 
\par         
A Taylor morphism is a differential morphism 
\begin{equation}\label{taylor}
(R,\, \{ \partial \sb 1 ,\, \partial \sb 2 ,\, \cdots ,\, \partial \sb d \})    
\to 
( S[[ X\sb 1,\, X\sb 2,\, \cdots ,\, X\sb d ]],\, \{ 
\frac{\partial }{\partial X\sb 1 }, \,
\frac{\partial }{\partial X\sb 2 }, \, \cdots ,  
  \frac{\partial }{\partial X\sb d } \}).   
\end{equation}
For a differential algebra  
$
(R,\, \{ \partial \sb 1 ,\, \partial \sb 2 ,\, \cdots ,\, \partial \sb d \}) , 
$
among Taylor morphisms \eqref{taylor}, 
 there exists the universal one $\iota \sb R$ given below. 
 \begin{definition}
 The universal Taylor morphism is a differential morphism 
 \begin{equation}\label{m28.1}
  \iota \sb R \colon (R,\,\, \{ \partial \sb 1 ,\, \partial \sb 2 ,\, \cdots ,\, \partial \sb d \}) 
 \to 
 ( R\n[[ X\sb 1,\, X\sb 2,\, \cdots ,\, X\sb d ]],\, \{ 
\frac{\partial }{\partial X\sb 1 }, \,
\frac{\partial }{\partial X\sb 2 }, \, \cdots ,  
  \frac{\partial }{\partial X\sb d } \})
\end{equation}
defined by the formal power series expansion 
\begin{equation*}
\iota\sb {R} (a) = \sum\sb {n \in \N ^d} \frac{1}{n!}
\partial ^n a \, X^n 
\end{equation*}
for an element $a \in R$, 
where we use  the standard  notation for multi-index. 
\par Namely,  
for $n= (n\sb 1,\,  n\sb 2, \cdots , n\sb d) \in \N ^d$, 
$$
 |n| = \sum \sb {i=1}^d n\sb i, 
 $$
 $$
\partial ^n = \partial \sb 1 ^ {n\sb 1}
\partial \sb 2 ^ {n\sb 2} \cdots 
\partial \sb d ^ {n\sb d}
$$
$$
n! = n\sb 1! n\sb 2!  \cdots n\sb d! 
$$
and 
$$
X^n = X\sb 1 ^{n\sb 1}X\sb 2 ^{n\sb 2}\cdots X\sb d ^{n\sb d}.
$$ 
 \end{definition}
 See Proposition (1.4) in Umemura \cite{ume96.2}.
 \subsubsection{The functor $\mathcal{F} \sb{ L/k}$ of infinitesimal deformations for a differential field extension}\label{9.25e}
For the partial differential 
field $L\s$, 
we have the universal Taylor morphism
\begin{equation}\label{m27.1}
\iota \sb {L\s} \colon L\s \to L\n [[ W\sb 1, W\sb 2, \cdots , W\sb d ]]
=L\n [[W]], 
\end{equation}
 where we replaced the variables $X$'s 
 in \eqref{m28.1}
 by the   variables $W$'s for a  notational reason.  
 The universal Taylor morphism \eqref{m27.1} gives a 
 differential morphism 
 \begin{equation}\label{m27.2}
 L\s[[X]] \to L\n [[ W\sb 1,\, W\sb 2, \cdots , W\sb d ]][[X]]. 
 \end{equation}
Restricting the morphism \eqref{m27.2} to the differential 
sub-algebra $\eL $ of $L\s[[X]] $, we get a differential 
morphism 
$\eL \to  L\n [[ W\sb 1,\, W\sb 2,\, \cdots , \, W\sb d ]][[X]] $
that we denote by $\iota $. So we have the 
differential 
morphism 
\begin{equation}\label{m27.4}
\iota \colon \eL \to 
L\n [[ W\sb 1,\,  W\sb 2,\,  \cdots ,\,  W\sb d ]][[X]].
\end{equation}
Similarly for every commutative $L\n$-algebra $A$, 
thanks to the differential morphism 
$$
L\n[[W]] \to A[[W]]
$$
we have the 
canonical differential morphism 
\begin{equation}\label{m27.3}
\iota \colon \eL \to A [[ W\sb 1,\,  W\sb 2,\,  \cdots ,\,  W\sb d ]] [[X]].
\end{equation}
We define  the functor 
$$
\mathcal{F}\sb {L/k} \colon (Alg/ L\n) \to (Set)
$$
from the category $(Alg/L\n)$  of commutative  $L\n$-algebras to the category 
$(Set)$ of sets, by associating to an $L\n$-algebra $A$,  the set of  
 infinitesimal deformations of the canonical morphism \eqref{m27.4}.  
So 
\begin{multline*}
\mathcal{F}\sb{L/k}(A)
= \{ f\colon \eL \to 
  A [[ W\sb 1,\,  W\sb 2,\,  \cdots ,\,  W\sb d ]][[X]] )\, | \, 
 f \,\text{\it \/ is a differential }
 \\
 \,\text{\it \/morphism 
 congruent to the canonical  morphism } \iota 
 \text{\it \/ modulo nilpotent elements} \\ 
 \text{\it \/such that } 
 f = \iota \,\text{\it \/ when restricted on the sub-algebra } \K 
\}.
 \end{multline*}
 \subsubsection{Group functor $\infgal(L/k)$ of infinitesimal automorphisms for a differential field extension} 
The Galois group in our Galois  theory is the group functor 
$$
\infgal (L/k) \colon (Alg/L\n ) \to (Grp)
$$
 defined 
 by 
 \begin{multline*}
 \infgal (L/k) (A) =
  \{
  \, 
 f \colon \eL \hat{\otimes} \sb {L\s} A[[W]] \to 
 \eL \hat{\otimes} \sb {L\s} A[[W]]
 \, | \, 
 f 
 \,\text{\it \/
  is a differential} \\ 
  \K \otimes \sb {L\s}A[[W]]   \text{\it-automorphism 
 continuous with respect to
  the $W$-adic topology} \\
  \,\text{\it \/ and congruent to the identity modulo nilpotent elements 
  }
    \}
 \end{multline*}
 for a commutative $L\n$-algebra $A$.
 Here the completion is taken with respect too the $W$-adic topology.
  See Definition 2.19 in \cite{mori09}. 
\par
Then the group functor $\infgal (L/k)$ operates on the functor 
$\mathcal{F}\sb {L/k}$ in such a way that 
the operation $(\infgal (L/k), \, \mathcal{F}\sb {L/k} )$ is a 
principal homogeneous space (Theorem 2.20, \cite{mori09}). 
\subsubsection{Origin of the group structure}\label{9.28a}
We explained the origin of the group functor $\infgal$. We illustrate it by an example.
\begin{example}. Let us consider a  differential  field extension 
$$
L/k:=\com (y), \, \gd ) /\com 
$$
such thqt $y$ is transcendental over the field $\com$ and 
\begin{equation}\label{9.27f}
\gd (y) = y \qquad \text{\it and }\qquad  \gd (\com) = 0 
\end{equation}
so that $k=\com $ is the field of constants 
of $L$.
\end{example}
The universal Taylor morphism 
$$
\iota \colon L \to L\s[[X]]
$$ 
maps $y\in L$ to 
$$
Y :=y\exp X \in L\s[[X]].
$$  
Since the field  extension $L\n /k\n =  \com (y)/\com$, taking $d/dy \in \mathrm{Der}(L\n/k\n
 $ as a basis of $1$-dimensional $L\n$-vector space 
$\mathrm{Der}(L\n/k\n)$,   we get 
$L\s := (L\n , \, d/dy)$. 
As we have a relation 
\begin{equation}\label{9.27a}
y\frac{\partial Y}{\partial y} =Y 
\end{equation}
that is an equality in the power series ring 
$L\s[[X]]$ 
so that the Galois hull $\eL /\K$ is 
\begin{equation}\label{9.27b}
\eL = \K[ \,\exp X\, ] , \quad \K =L\s \subset L\s[[X]] 
\end{equation}
by definition  of the 
Galois hull. 
\par
Now let us see the infinitesimal deformation functor $\mathcal{F}\sb{L/k}$. To this end,  
we Taylor expand the coefficients of the 
power series in $L\s [[X]]$
to get 
$$
\iota :L \to L\s[[X]] \to L\n[[W]][[X]]= L\n[[W, \, X]]
$$
so that  $\iota (y) = (y+W)\exp X \in L\n 
[[W, \, X]]$.

It follows from \eqref{9.27a} and \eqref{9.27b}, 
for a commutative 
$L\n $-algebra  $A$ an infinitesimal deformation 
$\varphi \in \mathcal{F} \sb{L/k} $ is determined by the image 
$$   
\varphi (Y) = c Y \in A[[W, \, X]], 
$$
where $c\in A$. Moreover any invertible element  $c\in A$ 
infinitesimally close to $1$  
defines an infinitesimal deformation so that we conclude
\begin{equation}\label{11.14a} 
\mathcal{F}\sb {L/k}(A) = \{ c\in A \, | \, c-1 \,\text{\it \/ is nilpotent} \}. 
\end{equation}
\par
{ \it Where does the group structure come from?} 
\par 
To see this, we have to look at the dynamical system defined by the differential equation \eqref{9.27f}. Geometrically the differential equation \eqref{9.27f} gives us a dynamicla system
on the line $\com $. 
$$
 y \mapsto Y=y\exp X 
  $$
  describes the dynamical system. 
  Observe the dynamical system trough 
  algebraic differential equations. is equivalent to considering the 
  deformations of the Galois hull. 
  So the (infinitesimal) deformation functor 
  measures the ambiguity of the observation. 
  In other words, the result due to our method 
  is \eqref{11.14a}. In terms of the initial condition, it looks as 
  $$
  y \mapsto cY\, |\sb {X=0} = cy\exp X\, |\sb{X=0}= cy.
  $$   
Namely, 
\begin{equation}\label{9.27d}
y \mapsto cy.
\end{equation}   
If we have  two transformations \eqref{9.27d}   
  $$
  y \mapsto cy, \qquad y \mapsto c^\prime y 
  $$
  the composite trasformation corresponds to the product  
  $$
   y\mapsto cc^\prime y.
   $$
\subsection{Difference
Galois theory}\label{0820a}
\par  
If we replace the universal Taylor morphism by the universal Euler morphism, we can 
construct  
a general Galois theory of difference equations (\cite{mori09}, 
\cite{morume09} ).
\par
\subsubsection{Universal Euler morphism}\label{0820b}
 Let $(R,\,  \sigma )$ be a difference ring so that $\sigma :R \to R$ is an endomorphism of a commutative 
 ring $R$. When there is no danger of confusion of the endomorphism 
 $\sigma$, we simply say the difference ring $R$ without referring to the endomorphism $\sigma$. We often have to talk however about the abstract ring $R$ that we denote by $R\n$.  
 For a commutative  ring $S$, we denote by $F(\N , \, S)$ the ring of functions 
 on 
 the set 
  $$
  \N = \{0, \, 1, \, 2, \cdots  \}
  $$ 
 taking values in the ring $R$. 
 For a function $f \in F(\N , \, S)$, we define the shifted function 
 $\Sigma f \in F(\N , \, S)$  by 
 $$
( \Sigma f)(n) = f( n + 1) \qquad \,\text{\it \/ for every } n \in \N.
 $$
 Hence  
 the shift operator 
 $$
 \Sigma :F(\N , \, S) \to F(\N , \, S)
 $$
 is an endomorphism of the ring  $F(\N , \, S)$ 
 so that $(F(\N , \, S), \, \Sigma )$ is a difference ring.  
\begin{remark}
In this paragraph \ref{0820a}
 and the next \ref{0820b},  
 in particular for the existence of the universal Euler morphism, 
 we do not need the 
 commutativity assumption of the 
underlying ring. 
\end{remark}
\par 
 Let $(R,\, \sigma)$ be a difference ring and $S$ a ring. 
 An Euler morphism is a difference morphism 
 \begin{equation}\label{1a2.1}   
 (  R, \, \sigma ) \to ( F ( \N , \, S), \, \Sigma ). 
 \end{equation}
 Given a difference ring $(R,\,  \sigma )$, 
 among the Euler morphisms \eqref{1a2.1}, there exists the universal one. 
 In fact,   for an element $a\in R$, we define the function $u[a] \in F (\N , \, R\n)$ 
 by 
 $$
u[ a](n) = \sigma ^n (a)  \qquad \,\text{\it \/ for } n \in \N .     
 $$ 
 Then the map
 \begin{equation}\label{9.19e}
\iota \colon  (R, \, \sigma ) \to 
( F(\N , \, R\n ), \, \Sigma ) \qquad a \mapsto u[a] 
\end{equation}
is the universal Euler morphism (Proposition 2.5, \cite{mori09}).  
\subsubsection{Galois hull $\eL /\K$ for a difference field extension $L/k$}\label{9.19g}
Let $(L, \,\sigma )/(k,\, \sigma  )$ be a difference field extension such that 
the abstract field $L\n $ is finitely
 generated over the abstract base field $k\n$. We constructed the Galois hull 
 $\eL / \K$ 
 as 
 in the differential case. Namely,   
we take a mutually commutative basis 
 $$
 \{ D \sb 1, \,  D \sb 2 , \cdots 
 , D\sb d \}
 $$
  of the $L\n$-vector space $\mathrm{Der} \, (L\n / k\n )$ 
 of $k\n$-derivations of the abstract field $L\n$. 
 We introduce the partial differential field 
$$
L\s := ( L\n , \, \{ D\sb 1 , \, D\sb 2, \,  \cdots , D\sb d    \} ). 
$$   
Similarly we define a differential structure on the  ring 
$F(\N , \, L\n)$ of functions taking values in $L\n$ by  considering     
the derivations $$\{D\sb 1 , \, D\sb 2, \,
 \cdots , D\sb d   \} .$$ In other words, 
 we work with the differential ring 
 $F(\N ,\, L\s )$.  
 So the ring 
 $F( \N , \, L\n )$ has a difference-differential structure 
 defined by the shift operator $\Sigma$ and the set 
 $$\{ D\sb 1 , \, D\sb 2, \,  \cdots , D\sb d    \}$$ 
 of derivations. Since there is no danger of confusion of the choice of the difference operator $\Sigma$, 
 we denote this difference-differential ring by 
 $$
 F(\N , \, L\s)    . 
 $$
We have the universal Euler morphism 
\begin{equation}\label{1m27.0.5}
\iota \colon L \to F (\N , \, L \n)
\end{equation}
that is a difference morphism. We added further the
$\{ D\sb 1 , \, D\sb 2, \,  \cdots , D\sb d    \}$-differential structure 
on $F(\N , \, L\n )$ or we replace the target space 
$F(\N , \, L\n )$ of the universal Euler morphism \eqref{1m27.0.5}
 by $F(\N , \, L\s )$ so that we have 
 $$
 \iota \colon L \to F ( \N , \, L\s ). 
 $$
\par 
In Definition \ref{1a4.6} below, we work in the difference-differential ring 
$F(\N , \, L\s)$ with difference operator 
$\Sigma $ and differential operators  
$
\{ D\sb 1 , \, D\sb 2, \,  \cdots , D\sb d    \}.
$
We identify  with $L\s$ the set of constant  functions on $\N$. 
Namely, 
$$
L \s = \{ f \in F (\N , L\s ) \, | \, f(0 ) = f(1) = f(2) = \cdots \in L\s \}.
$$
Therefore $L\s$ is a difference-differential sub-field of the 
difference-differential 
ring $F ( \N , \, L\s )$. The action of the  shift  operator 
on $L\s$ 
 being trivial, the notation is adequate. Similarly, 
we set  
$$
k \s := \{ f \in F (\N , L\s ) \, | \, f(0 ) = f(1) = f(2) = \cdots \in   k \subset L\s \}.
$$
So both the shift operator and the derivations act trivially on 
$k\s$ and so $k\s$ is a difference-differential sub-field of $L\s$ and hence 
of the difference-differential algebra $F(\N , \, L\s )$. 
\begin{definition}\label{1a4.6}
The Galois hull $\eL / \K $ is a difference-differential sub-algebra extension 
of $F( \N , \, L\s)$, where 
$\eL$ is the difference-differential sub-algebra generated by the 
image $\iota (L)$ and $L\s$ 
and $\K $ is the sub-algebra generated by the image $\iota (k ) $ 
and $ L\s$. So $\eL / \K$ is a difference-differential algebra  extension with difference operator $\Sigma $ and derivations 
$
\{ D\sb 1 , \, D\sb 2, \,  \cdots , D\sb d    \}.
$
\end{definition} 
 \subsubsection{The functor $\mathcal{F} \sb{L/k}$ of infinitesimal deformations
 for a difference field extension}\label{9.25d}
For the partial differential 
field $L\s$, 
we have the universal Taylor morphism
\begin{equation}\label{1m27.1}
\iota \sb {L\s} \colon L\s \to L\n [[ W\sb 1, W\sb 2, \cdots , W\sb d ]]
=L\n [[W]]. 
\end{equation}
 The universal Taylor morphism \eqref{1m27.1} gives a 
 difference-differential morphism 
 \begin{equation}\label{1m27.2}
 F(\N , L\s ) \to F(\N ,  L\n [[ W\sb 1, W\sb 2, \cdots , W\sb d ]] ). 
 \end{equation}
Restricting the morphism \eqref{1m27.2} to the difference-differential 
sub-algebra $\eL $ of $F(\N , L\s )$, we get a difference-differential 
morphism 
$\eL \to  F(\N , L\n [[ W\sb 1, W\sb 2, \cdots , W\sb d ]] )$
that we denote by $\iota $. So we have the 
difference-differential 
morphism 
\begin{equation}\label{1m27.4}
\iota \colon \eL \to 
F(\N ,  L\n [[ W\sb 1, W\sb 2, \cdots , W\sb d ]] ).
\end{equation}
Similarly for every commutative $L\n$-algebra $A$, 
thanks to the differential morphism 
$$
L\n[[W]] \to A[[W]], 
$$
we have the 
canonical difference-differential morphism 
\begin{equation}\label{1m27.3}
\iota \colon \eL \to F(\N ,  A [[ W\sb 1, W\sb 2, \cdots , W\sb d ]] ).
\end{equation}
We define  the functor 
$$
\mathcal{F}\sb {L/k} \colon (Alg/ L\n) \to (Set)
$$
from the category $(Alg/L\n)$  of commutative $L\n$-algebras to the category 
$(Set)$ of sets, by associating to a commutative $L\n$-algebra $A$,  the set of  
 infinitesimal deformations of the canonical morphism \eqref{1m27.4}.  
So 
\begin{multline*}
\mathcal{F}\sb{L/k}(A)
= \{ f\colon \eL \to 
F(\N,  \, A [[ W\sb 1, W\sb 2, \cdots , W\sb d ]] )\, | \, 
 f \,\text{\it \/ is a differential }
 \\
 \,\text{\it \/morphism 
 congruent to the canonical  morphism } \iota 
 \,\text{\it \/ modulo nilpotent elements} \\ 
 \,\text{\it \/such that} 
 f = \iota \,\text{\it \/ when restricted on the sub-algebra } \K 
\}.
 \end{multline*}
See Definition 2.13 in \cite{mori09},  
 for a rigorous definition. 
 \subsubsection{Group functor $\infgal(L/k)$ of infinitesimal automorphisms for a difference field extension} 
The Galois group in our Galois  theory is the group functor 
$$
\infgal (L/k) \colon (Alg/L\n ) \to (Grp)
$$
 defined 
 by 
 \begin{multline*}
 \infgal (L/k) (A) =
  \{
  \, 
 f \colon \eL \hat{\otimes} \sb {L\s} A[[W]] \to 
 \eL \hat{\otimes} \sb {L\s} A[[W]]
 \, | \, 
 f 
 \,\text{\it \/
  is a difference-differential} \\ 
  \K \otimes \sb {L\s}A[[W]]   \,\text{\it \/-automorphism 
 continuous with respect to
  the $W$-adic topology} \\
  \,\text{\it \/ and congruent to the identity modulo nilpotent elements 
  }
    \}
 \end{multline*}
 for a commutative  $L\n$-algebra $A$. 
 Here the completion is taken with respect to the $W$-adic topology.
 See Definition 2.19 in \cite{mori09}. 
\par
Then the group functor $\infgal (L/k)$ operates on the functor 
$\mathcal{F}\sb {L/k}$ in such a way that 
the operation $(\infgal (L/k), \, \mathcal{F}\sb {L/k} )$ is a 
principal homogeneous space (Theorem2.20, \cite{mori09}).   
\subsection{Introduction of more precise  notations}
\label{9.25a}
So far, we explained general differential 
Galois theory and general difference Galois theory. 
To go  further we have to make our notations more precise.
\par
For example, we defined the Galois hull for a differential field extension 
in Definition \ref{a4.6} and the
Galois hull for a difference field extension in Definition \ref{1a4.6}. 
Since they are defined by the same principle, we denoted both of them by $\eL / \K$.  
 We have to, however, distinguish them. 
 \begin{definition} 
We denote the Galois hull
for a differential field extension by $\eL \sb \gd /\K \sb\gd$
and we use the symbol $\eL \sb \sigma / \K \sb \sigma $ 
for the Galois hull of a difference field extension. 
\end{definition}
We also have to distinguish the functors $\mathcal{F}\sb {L/k}$ 
and $\infgal 
(L/k)$
in the differential case and in the difference case: we add the suffix $\gd$ for the differential case 
and the suffix $\sigma$ for the difference case so that 
\begin{enumerate}
\item 
We use 
$\mathcal{F}\sb {\gd  L/k}$ and $\infgal \sb \gd
(L/k)$ when we deal with differential algebras.
\item 
We use $\mathcal{F}\sb {\sigma L/k}$ and $\infgal \sb \sigma 
(L/k)$ for difference algebras.
\end{enumerate}
\par 
We denoted, according to our convention, 
for a commutative algebra $A$ the category of commutative $A$-algebras by 
$(Alg/A)$. As we are going to consider the category of not necessarily commutative 
$A$-algebras. This notation is confusing. 
So we clarify the notation.
\begin{definition} 
 We often denote the category of commutative 
$A$-algebras by $(CAlg/A)$.
\end{definition}
\par
\section{Hopf Galois theory}
Picard-Vessiot theory is a Galois theory of linear differential or difference equations. 
The idea of introducing Hopf algebra in Picard-Vessiot theory is traced back to Sweedler \cite{swe69}. 
Specialists in Hopf algebra succeeded in unifying 
Picard-Vessiot theories for differential equations and difference equations \cite{amaetal09}. They further succeeded in generalizing the Picard-Vessiot 
theory for difference-differential equations, where the operators are 
not necessarily commutative. 
 Heiderich \cite{hei10} combined the idea of 
Picard-Vessiot  theory via Hopf algebra with our general  
Galois theory for non-linear equations \cite{ume96.2}, \cite{mori09}.
His general theory includes a wide class of difference and differential algebras. 
\par
There are two major advantages in his theory.  
\begin{enumerate}
\renewcommand{\labelenumi}{(\arabic{enumi})}
\item Unified study of non-linear 
differential equations and difference equations.
\item Generalization of universal Euler morphism and Taylor morphism.
\end{enumerate}
\par
Let $C$ be a field. For $C$-vector spaces $M, \, N$, we denote 
by $\sb  C \mathbf{M} (M, N)$ the set of $C$-linear maps from $M$ to $N$. 
\begin{example}\label{9.18a}
Let $\calH : = C[\G\sb a] = C [t]$ be the $C$-Hopf algebra of the coordinate 
ring of the additive group scheme $\G\sb {a C}$ over the field $C$. Let $A$ be a commutative $C$-algebra 
and 
$$
\Psi\in \, \sb C\M ( A\otimes\sb C \calH, A)      
= \, \sb C \M (A, \, \sb C\M (\calH , A))
$$ 
so that $\Psi$ defines two $C$-linear maps 
\begin{enumerate}
\renewcommand{\labelenumi}{(\arabic{enumi})}
\item $\Psi\sb 1 \colon A\otimes\sb C \calH \to  A$,
\item $\Psi \sb 2 :A \to  \, \sb C\M (\calH , A)$.
\end{enumerate}
\begin{definition}
We say that $(A, \Psi )$ is an $\calH$-module algebra if the following equivalent  conditions are satisfied. 
\begin{enumerate}
\renewcommand{\labelenumi}{(\arabic{enumi})}
\item The $C$-linear map $\Psi\sb 1 \colon A\otimes\sb C \calH \to  A$
defines  an operation of the $C$-algebra  
$\calH$ on the $C$-algebra $A$, 
\item The $C$-linear map $$\Psi \sb 2 :A \to  \, \sb C\M (\calH , A)$$
is a $C$-algebra morphism, the dual 
$\sb C\M (\calH , A )$ 
of co-algebra $\calH$ being a $C$-algebra.    
\end{enumerate} 
\end{definition}
Concretely the dual algebra $\sb C\M (\calH , A)$ is the formal power series ring $A[[X]]$.
\par It is a comfortable  exercise to examine
that $( A, \Psi )$ is an  
$\calH$-module algebra if and only if 
$A$ is a differential algebra with derivation $\gd$ such 
that $\gd (C) =0$. 
 When the equivalent conditions are satisfied,  for every element $a$ in the algebra  $A$, 
$\Psi (a\otimes t) = \gd (a)$ and   
the $C$-algebra morphism  
$$
\Psi \sb 2 :A \to  \, \sb C\M (\calH , A) =A[[X]]
$$
is the universal Taylor morphism. 
So 
$$
\Psi\sb{2}(a) = \sum\sb{n=0}^{\infty} \frac{1}{n!}\delta^{n}(a)X^{n}\,\, \in A[[X]]
$$ 
for every $a \in A$. 
See Heiderich \cite{hei10}, 2.3.4.
\end{example}
In Example \ref{9.18a},  we explained the differential case.
If we take an appropriate {\it bialgebra\/} for $\calH$, we get difference structure and the universal 
Euler morphism. See \cite{hei10}, 2.3.1.  
More generally we can take any bialgebra $\mathcal{H}$ to get 
an algebra $A$ with operation of the algebra $\mathcal{H}$ and a 
morphism 
$$
\Psi \sb 2 :A \to  \, \sb C\M (\calH , A) 
$$
generalizing the universal 
Taylor morphism and Euler morphism. So we can 
define the Galois hull $\eL /\K $ and develop 
a general Galois theory for a field extension $L/k$ 
with operation of the algebra $\mathcal{H}$. 
In the differential case as well as in the difference case, the corresponding bialgebra $\mathcal{H}$ is co-commutative so that the dual algebra $ \sb C \mathrm {M}(\mathcal H,  \, A)$
is a commutative algebra. Consequently the 
Galois hull $\eL /\K$ that are sub-algebras in the 
commutative algebra 
 $ \sb C \mathrm {M}(\mathcal H,  \, A)$.
 In these case the Galois hull is an algebraic counter part of 
the geometric object, algebraic Lie groupoid. See Malgrange \cite{mal01}. Therefore 
the most fascinating queion is  
 \begin{question}
Let us consider a non-co-commutative bialgebra 
$\mathcal{H}$ and assume that the Galois hull $\eL/\K$ that 
is a sub-algebra of the dual algebra 
$ \sb C \mathrm {M}(\mathcal H,  \, A)$, is not a commutative algebra. 
Does the Galois hull $\eL / \K$ quantize the 
algebraic Lie groupoid?
\end{question}      
We answer affirmatively  the question by analyzing examples in 
\qsi field extensions.    
\begin{remark} 
How non-co-commutative the bialgebra $\mathcal{H}$ may be, 
so far as one considers  linear equations, the Galois hull  
$\eL / \K$ is a commutative sub-algebra of the 
non-commutative algebra 
$ \sb C \mathrm {M}(\mathcal H,  \, A)$. 
Hence one does not encounter quantum groups,  except for linear algebraic groups, studying generalized  
Picard-Vessiot theories. See Hardouin \cite{har10} and   Umemura \cite{ume11}. 
\end{remark}
 %
\par 
Let $C$ be a field, $q$  an element of $C$.
We use a standard notation of $q$-binomial coefficients. 
To this end, let $Q$ be a variable over the field $C$.  
\par 
We set $[n]\sb Q = \sum \sb{i=0} ^ {n-1} Q^i \in C[Q]$ for positive integer 
$n$. We need also $q$-factorial 
$$
[n]\sb Q ! :=\prod \sb{i=1} ^n [i]\sb Q \qquad 
\text{ for a positive integer $n \qquad$  and } \, 
[0]\sb Q ! := 1 .  
$$ 
So $[n]\sb Q \in C[Q]$.
The $Q$-binomial coefficient is defined for $m, n \in \N $ by 
$$
\binom{m}{n}\sb Q = \begin{cases}
\frac{[m]\sb Q!}{[m-n]\sb Q ![n]\sb Q!} & \text{ if } m \ge n, \\
0 & \text{ if } m < n.
\end{cases}
$$
Then we can show that the rational function 
$$
\binom{m}{n}\sb Q \in C(Q)
$$
is in fact a polynomial or 
$$
\binom{m}{n}\sb Q \in C[Q].
$$
We have a ring morphism 
\begin{equation}\label{9.17c}
 C[Q] \to C[q], \qquad  Q \mapsto q
 \end{equation}
  over $C$ and we denote 
the image of the polynomial 
$$\binom{m}{n}\sb Q $$ under  morphism \eqref{9.17c} by 
$$
\binom{m}{n}\sb q.
$$  
\par

\subsection{$q$-skew iterative $\sigma$-differential algebra \cite{har10}, \cite{hay08} }
First non-trivial example of a  Hopf Galois theory dependent  on 
non-co-commutative Hopf algebra is Galois theory of $q$-skew iterative $\sigma$-differential field  extensions,   
abbreviated as \qsi field extensions.
\subsubsection{Definition of \qsi algebra}
\begin{definition}\label{a3.3}
Let $C$ be a field of characteristic $0$ and $q\not= 0$ an element of the field $C$.   
A $q$-skew iterative $\sigma $-differential algebra 
$( A, \,\sigma ,\,  \theta ^* ) = ( A, \sigma , \{ \theta ^{(i)}\} \sb{i\in \N} )$, a   
\qsi algebra for short, 
 consists of a  
$C$-algebra $A$ that is eventually non-commutative,
a $C$-endomorphism $\sigma :A \to A$ 
of the $C$-algebra $A$ and a family 
$$
\theta ^{(i)}\colon A \to A \qquad \text{ for $i \in \N$} 
$$
of $C$-linear maps 
satisfying the following conditions.
\begin{enumerate}
\renewcommand{\labelenumi}{(\arabic{enumi})}
\item    $\theta ^{(0)} = \Id\sb A$, 
\item $\theta ^{(i)}\sigma    = q^i \sigma \theta ^{(i)} \qquad \text{ for every } i \in \N$, 
\item  $\theta ^{(i)}(ab)   =
\sum\sb{l+m= i} 
\sigma^{m}
(
\theta ^{(l)}(a) 
) 
\theta ^{(m)}(b)$, 
\item $\theta ^{(i)}\circ \theta ^{(j)} =\binom{i+j}{i}\sb q \theta ^{(i+j)}$.
\end{enumerate}
We say that 
an element $a$ of the \qsi algebra $A$ is a constant if  
$\sigma (a) = a$ and $\theta ^{(i)}(a) = 0$ for every $i\ge 1$. 
\par
A morphism of \qsi 
$C$-algebra morphism compatible with the endomorphisms $\sigma$ and the derivations 
$\theta ^ *$. 
\end{definition} 
\par 
Both differential algebras and difference algebras are \qsi algebras.

 \subsubsection{Difference algebra and a \qsi algebra}\label{9.19c}
 Let $A$ be a commutative $C$-algebra and $\sigma :A \to A$ be a $C$-endomorphism of the ring $A$. So $( A , \sigma )$ is a difference algebra. 
 If we set $\theta ^{(0)} = \Id \sb A$ and 
 $$
 \theta ^{(i)} (a) = 0 \text{ for every element $a\in A$ and for } i =1,\,2, \, 3, \, \ldots .
 $$ 
 Then $(A, \, \sigma , \, \theta ^ *)$ is a 
 \qsi algebra. 
 \par 
 Namely we have a functor of the category $(Dif\!f'ce Alg/C)$ of $C$-difference algebras to the category $(q\text{-}SI \sigma \text{-}dif\!f'ial Alg/C)$ of 
 \qsi algebras over $C$: 
\begin{equation*}
(Dif\!f'ce Alg/C) \rightarrow (q\text{-SI} \sigma \text{-}dif\!f'ial Alg/C). 
\end{equation*}  
 \par 
 Let $t$ be a variable over the field $C$ and 
 let us now assume
\begin{equation}\label{8.8a}
q \not= 1 \qquad \text{for every integer } n \in \N .
\end{equation}
  We denote by $\sigma \colon C(t) \to C (t)$ the $C$-automorphism of the rational function field $C(t)$ sending the variable $t$ to $qt$. 
We consider a difference algebra extension $(A, \, \sigma )/(C(t), \, \sigma )$. 
If we set 
$$
\theta ^{(1)}(a)= \frac{\sigma (a) - a}{(q-1)t} \qquad 
\text{for every element }a \in A
$$   
and 
$$
\theta ^{(i)}=\frac{1}{[i]\sb q !}{\theta ^{(1)}}^i 
\qquad 
\text{for } i=2, \, 3, \, \ldots .
$$
 Then $(A, \sigma , \theta^{*} )$ is a \qsi algebra. Therefore if $q \in C$ satisfies \eqref{8.8a}, then we have a functor 
\begin{equation}
(Dif\!f'ce Alg/(C(t),\sigma)) \rightarrow (q\text{-SI} \sigma \text{-}dif\!f'ial Alg). 
\end{equation}
 \subsubsection{Differential algebra and \qsi algebra}
 \label{9.19b}
 Let $(A,\,  \theta )$ be a differential algebra such that 
 the field $C$ is a subfield of the ring $C\sb A$ of constants of the differential algebra $A$. We set 
\begin{align*}
\theta ^{(0)} &= \Id \sb A, \\
 \theta ^{(i)}&= \frac{1}{i!}\theta ^ i\qquad  \text{for } i = 1, \, 2, \, 3, \ldots .
\end{align*}
 Then $(A, \, \Id \sb A ,\,  \theta ^ *)$ is a \qsi algebra for $q=1$. 
 In other words, we have a functor 
 $$
 (Diff'ialAlg/C) \to (q\text{-}SI \sigma \text{-}diff'ial Alg/C)
 $$
 of the category of (\,commutative\,)  differential $C$-algebras to the category of 
 \qsi algebras over $C$. 
  We have shown that both difference algebras and differential algebras are 
 particular instances of \qsi algebra.
 
 \subsubsection{Example of \qsi  
 algebra \cite{hei10}}\label{10.3a}
 We are going to see  
 \qsi algebras on the border between  commutative algebras and non-commutative algebras. The example below seems to suggest 
  that 
 it looks natural to seek \qsi algebras in the category of non-commutative algebras. 
 \par
 An example of \qsi algebra arises from a commutative 
 $C$-difference algebra $(S, \, \sigma)$. We need, however, a non-commutative ring,   
  the twisted power series ring 
  $(S, \, \sigma )[[X]]$ over the difference ring $(S, \sigma )$ that has a natural \qsi algebra structure. 
 \par
 Namely,   let $(S, \sigma )$ be the $C$-difference ring so that 
 $\sigma :S \to S$ is a $C$-algebra endomorphism of the commutative  ring $S$. 
 We introduce the following twisted formal power series ring 
 $(S, \, \sigma ) [[X]]$
 with coefficients in $S$ 
 that is the formal power series ring $S[[X]]$ 
 as an additive group 
 with the following commutation relation 
 $$
  aX = X \sigma (a)  \qquad \text { for every $a\in S$}. 
  $$
  So more generally  
  \begin{equation}\label{a11.1}
   aX^n  = X^ n \sigma ^n (a) 
   \end{equation}
   for every $n \in \N$. 
   The multiplication of two formal power series is defined by extending \eqref{a11.1} by linearity. 
   Therefore the twisted formal power series ring $(S, \sigma )[[X]])$ is non-commutative in general. 
   By commutation relation \eqref{a11.1}, 
   we can identify 
   $$
   (S, \sigma  )[[X]] = \{\sum \sb{i=0}^\infty X^i a\sb i \, |\, 
   a\sb i \in S  \text{ for every } i\in \N \}
   $$
   as additive groups.  
  \par 
  We are going to see that 
  the twisted formal 
   power series ring 
   has a natural 
   \qsi structure. We define first a ring  endomorphism  
$$
\hat{\Sigma} \colon (S, \, \sigma ) [[X]] \to 
  (S, \, \Sigma )[[X]]
  $$
    by setting 
  \begin{equation}\label{a5.1}
  \hat{\Sigma} (\sum \sb{i=0}^\infty X^i a\sb i ) = \sum 
  \sb{i=0}^\infty X^i q ^i\sigma (a\sb i)   \qquad \text{ for every } i \in \N , 
  \end{equation} 
  for every element 
  $$
  \sum \sb{i=0}^\infty  X^i a\sb i \in (S,
   \sigma )[[X]]
   $$. 
The operators  
 $\Theta ^* =\{ \theta ^{(l)} \}\sb{
  l\in \N}$ are defined by 
  \begin{equation}\label{a4.5}
  \Theta ^{(l)}(\sum \sb{i=0}^\infty  X^ia\sb i  ) = 
  \sum \sb{i=0}^\infty X^i \binom{i+l}{l}\sb q  a\sb {i+ l}
  \qquad \text{for every }i \in \N. 
  \end{equation}
   Hence the twisted formal power series ring $(S, \, \sigma ) [[X]], \hat{\Sigma} , \Theta ^* ) $ is a non-commutative \qsi 
  ring. We denote this  \qsi ring simply by $(S, \sigma )[[X]]$.
 See \cite{hei10}, 2.3.  
  In particular, if we take as the coefficient difference ring $S$ 
  the difference ring 
$$(F(\N , A), \, \Sigma )$$
  of functions on $\N$ taking values in a ring $A$ in \ref{0820a}, where 
  $$
  \Sigma :F(\N , \, A ) \to F(\N , \, A)
  $$ 
  is the shift operator,  we obtain the \qsi ring 
 \begin{equation}\label{10.1a}
(  \left( F( \N , A ), \Sigma ) [[X]], \, \hat{\Sigma}
, \, \Theta ^* \right) .
 \end{equation}
\begin{remark}
We assumed that the coefficient difference ring 
$(S, \, \sigma)$ is commutative. The commutativity  assumption on the ring $S$ is not necessary. 
Consequently we can use non-commutative ring $A$ in \eqref{10.1a}
\end{remark}

\subsubsection{Universal Hopf morphism for a \qsi algebra}   
We introduced in \ref{0820a} the 
difference 
ring of functions $(F(\N , \, A), \,\Sigma )$ on the set $\N$ taking values in a ring $A$.  
It is useful to denote the function f by a matrix
\begin{equation*}
\left[\begin{array}{cccc}
0 & 1 & 2 & \cdots \\
f(0) & f(1) & f(2) & \cdots
\end{array}\right] .
\end{equation*}
\par 
Eor an element $b$ 
of a 
difference algebra $(R,\, \sigma )$ or a 
\qsi algebra 
$( R, \sigma , \, \theta ^* )$, we denote 
 by 
$u[b]$ a function on $\N$ taking values in the abstract ring $R\n$
such that 
$$
u[b](n) = \sigma ^ n (b) \qquad \text{for every } n \in \N 
$$
so that 
\begin{equation*}
u[b] = \left[\begin{array}{cccc}
0 & 1 & 2 & \cdots \\
b & \sigma^1 (b) & \sigma^2 (b) & \cdots
\end{array}\right] .
\end{equation*}

So $u[b] \in F(\N , R\n )$.

   \begin{proposition}[Proposition 2.3.17,  Heiderich \cite{hei10}]
  \label{a4.1}
   For a \qsi algebra $(R, \, \sigma , \, \theta ^* )$, hence in particular  for  
   an iterative $q$-difference ring $R$, there exists 
   a canonical morphism, which we call the universal Hopf 
    morphism 
\begin{equation}\label{9.19a}
\iota \colon (R, \, \sigma , \, \theta ^* ) 
\to 
 \left((F(\N , R\n ), \Sigma  )[[X]],\, \hat{\Sigma}, \, \hat{\Theta} ^* \right),   
 \qquad a \mapsto \sum  \sb{i=0}
^\infty  X^i u[\theta ^{(i)}(a)] 
\end{equation}
of \qsi algebras. 
 \end{proposition}   %
\par 
 We can also characterize the universal Hopf morphism as the 
 solution of a universal mapping property. 
 \par 
 When $q=1$ and $\sigma =\Id \sb R$ 
 and $R$ is commutative  
 so that 
 the \qsi ring $(R, \, \Id\sb R , \, \theta ^* )$ is  simply a 
 differential algebra as we have seen in \ref{9.19b}, the universal Hopf morphism \eqref{9.19a}    is the universal 
 Taylor morphism in \eqref{9.19d}. 
 Similarly a commutative difference ring is a \qsi algebra 
 with trivial derivations as we noticed in \ref{9.19c}. In this case the universal Hopf morphism \eqref{9.19a} is nothing but 
 the universal 
 Euler morphism \eqref{9.19e}.   
 Therefore the
 universal Hopf morphism unifies the universal Taylor morphism and the Universal Euler morphism.  
 \par 
 Let us recall  the following fact. 
 \begin{lemma}
 Let $(R, \,  \sigma , \, \theta ^* )$ be a \qsi domain. 
 If the endomorphism $\sigma \colon R \to R$ is an automorphism, then 
 the field $Q(R)$ of fractions of $R$ has the unique  structure of 
 \qsi field.  \par 
 If moreover $R$ is an \iqd algebra,  
  then the field $Q(R)$ of fractions of $R$ is also an \iqd field.       
 \end{lemma}
 \begin{proof}
 See for example, Proposition 2.5 of \cite{hay08}.  
 \end{proof}
\par 
We can interpret  the Example in \ref{10.3a} from another 
view point. We constructed there from a difference ring 
$(S, \, \sigma )$  a \qsi algebra $((S, \, \sigma )[[X]], \, 
\hat{\Sigma} , \, \hat{\Theta} ^* ) $. We notice that this procedure 
is a particular case of Proposition \ref{a4.1}. 
In fact, given a difference ring $(S, \, \sigma) $. So as in \ref{9.19c}, by adding the trivial derivations, we 
get the \qsi algebra $(S, \, \sigma , \, \theta ^*)$, where 
\begin{align*} 
\theta ^{(0)} &= \Id \sb S, \\ 
\theta ^{(i)} &= 0 \qquad \text{ for }i \ge 1. 
\end{align*}
Therefore we have the universal Hopf morphism 
$$
(S, \, \sigma , \, \theta ^*) \to (F(\N , \, S\n )[[X]], \hat{\Sigma} , \, \hat{\Theta} ^* )  
$$
by Proposition \ref{a4.1}.
So we obtained the \qsi algebra $ (F(\N , \, S\n )[[X]], \, \hat{\Sigma},  , \, \hat{\Theta} ^* ) $ as a result of composite of 
two functors. Namely, 
\begin{enumerate}
\renewcommand{\labelenumi}{(\arabic{enumi})}
\item 
The functor$\colon $(\, Category of Difference algebras\,) $\to$
 (\,Category of \qsi algebras\,) of adding trivial derivations
\item
The functor $\colon$ (\, Category of \qsi algebras\,) $\to $ 
(\,Category of 
\qsi algebras\,),  $A \mapsto B$ if there exists the universal
Hopf morphism $A \to B$.
\end{enumerate}

 \subsubsection{Galois hull 
$ \eL /\K$ for a \qsi field extension}\label{9.25c} 
 We can develop a general Galois theory for 
 \qsi
  field extensions analogous to our theories in \cite{ume96.1}, \cite{ume06} and \cite{ume07} thanks to the universal Hopf morpism.  
 Let $L/k$ be an extension of 
 \qsi 
  fields such that 
 the abstract field $L\n$ is finitely generated over the abstract field $k\n$. 
 Let us assume that we are in characteristic $0$. General theory in 
 \cite{hei10} 
 works, however, also in characteristic $p\ge 0$. 
We have by Proposition \ref{a4.1} the universal Hopf  morphism 
\begin{equation} 
 \iota :(L, \, \sigma , \, \theta ^* ) \to \left( (F(\N , L\n ), \Sigma )[[X]], 
\,  \hat{\Sigma},\,  \hat{\Theta} ^* \right)
 \end{equation}
 so that the image $\iota (L)$ is a copy of 
 the \iqd field $L$. We have another copy of $L\n$. 
The set 
\begin{multline}\label{a1.1}
 \{ f = \sum \sb {i=0}^\infty X^i a\sb i  \in F(\N , L\n )[[X]] \, | \, 
a\sb i = 0 \text{ for every } i  \ge 1 \text{ and }
 \, \Sigma (a\sb 0 )= a\sb 0
  \} \\
  =\{ f \in F(\N , L\n )[[X]] \, | \, \hat{\Sigma}(f) = f , \, \hat{\Theta}^{(i)} (f) = 0 \text{ for every } i \ge 1 \} 
\end{multline}
forms the sub-ring 
of constants in the
\qsi algebra of the 
 twisted power series 
 $$\left(
 (
 F(\N , \, L\n ) , \Sigma 
 )
[[X]],\, \hat{\Sigma}, \, \hat{\Theta}^* 
\right) .
$$ 
We identifyy $L\n$ with the ring of constants through
the following morphism. 
For an element $a\in L\n$, we denote the constant function $f\sb a$ on $\N$ taking the value $a \in L\n$ 
so that 
\begin{equation}\label{9.24a}
L\n \to  \left(
 (
 F(\N , \, L\n ) , \Sigma 
 )
[[X]],\, \hat{\Sigma}, \, \hat{\Theta}^* 
\right) 
    , \qquad a \mapsto f\sb a
\end{equation}
is an injective  ring morphism.
We may denote the sub-ring in \eqref{a1.1} by $L\n$. In fact, as an abstract ring 
it is isomorphic to the abstract field $L\n$ and the endomorphism $\hat{\Sigma}$  and  
the derivations $\hat{\Theta} ^{(i)}, \,  (i \ge 1 )$ operate trivially on the 
sub-ring. 
 \par 
 We are now exactly in the same situation as in 
\ref{9.19h} of the differential case and in \ref{9.19g} of  the difference case.
 We choose a mutually commutative basis $\{ D\sb 1, D\sb 2, \cdots , D\sb d \}$ of the 
 $L\n$-vector space $\mathrm{Der}(L\n / k\n )$ of $k$-derivations.   
 So $L\s := (L\n,   \{ D\sb 1, D\sb 2, \cdots , D\sb d \})$ is a 
differential field.  
\par 
So we introduce derivations $D\sb 1,\,  D\sb 2 , \, \cdots,\,  D\sb d $ 
operating on the coefficient ring $F(\N , \, L\n )$. In other words,
we replace the target space 
$F(\N , \, L\n )[[X]]$ by   
 $F(\N ,\,  L\s )[[X]]$.  
Hence the universal Hopf morphism in Proposition \ref{a4.1} becomes 
$$
\iota :L \to F(\N , \, L\s ) [[X]]).    
  $$
In the twisted formal power series ring
$(F(\N , L\s)[[X]], \hat{\Sigma}, \, \hat{\Theta} ^* )$, 
we add differential operators 
$$
D\sb 1, D\sb 2 , \cdots, D\sb d .
$$
So  we have a set $\mathcal{D}$ of the following operators on the ring 
$(F(\N , L\s ), \Sigma )[[X]] $.
\begin{enumerate}
\renewcommand{\labelenumi}{(\arabic{enumi})}
\item The endomorphism $\hat{\Sigma}$. 
$$
\hat{\Sigma} (\sum \sb {i=0}^\infty  X^ i a\sb i ) = \sum \sb {i=0}^
\infty 
 X^i q^ i (\Sigma ( a\sb i )),
$$
$\Sigma :F(\N , L\s ) \to F(\N , L\s )$ being the shift operator of the  ring of functions on $\N$. 
\item  
The $q$-skew  $\hat{\Sigma}$-derivations $\hat{\Theta} ^{(i)}$'s in \eqref{a4.5}.  
\begin{equation*}
  \hat{\Theta} ^{(l)}(\sum \sb{i=0}^\infty X^i a\sb i   ) = 
  \sum \sb{i=0}^\infty 
 X^i \binom{l+ i}{l}\sb q a\sb {i+ l} \qquad  \text{for every } l \in \N .
 \end{equation*}
\item 
The derivations $D\sb 1,\, D\sb 2 ,\, \cdots,\, D\sb d $ 
operating 
through the coefficient ring $F(\N , L\s )$ as in \eqref{9.19a}.   
  \end{enumerate} 
 \par 
 Hence we may write $F(\N , L\s ), \, \mathcal{D} )$, where 
 $$
 \mathcal{D} = \{ \hat{\s\Sigma} , \, D\sb 1 ,\  D\sb 2,\, \cdots , \, D\sb d, \, \hat{\Theta} ^ *    \}  \text{ and } \hat{\Theta} ^ *
  = \{  
 \hat{\Theta} ^{(i)} \} \sb {i \in \N }.  
 $$
 \par
We identify using inclusion \eqref{9.24a}
$$
L\s \to F((\N ,\, L\s )[[X]]. 
$$
  We sometimes denote the image $a\sb f$ of an element $a\in L\s$ by 
  $a\s$. 
 \par
  We are ready to define Galois hull as in Definition \ref{a4.6}. 
\begin{definition}
The Galois hull $\eL / \K $ is a $\mathcal{D}$-invariant sub-algebra 
of $F( \N , \, L\s)[[X]]$,  
 where 
$\eL$ is the $\mathcal{D}$-invariant sub-algebra generated by the 
image $\iota (L)$ and $L\s$ 
and $\K $ is the 
$\mathcal{D}$-invariant sub-algebra generated by the image $\iota (k ) $ 
and $ L\s$. So $\eL / \K$ is a $\mathcal{D}$-algebra  extension. \par 
As in \ref{9.25a}, if we have to emphasize that we deal with \qsi-algebras, we denote the Galois hull by 
$\eL\sb {\sigma\, \theta }/
\K \sb{\sigma \, \theta}$.
\end{definition} 
 \par 
We notice  that we are now in a totally new situation. 
In the differential case, the universal 
Taylor morphism maps the given fields to the commutative algebra  
of the formal power series ring so that 
the 
Galos hull is an extension of commutative algebras. 
Similarly for the universal Euler morphism of a  difference rings. The commutativity of the Galois hull comes from the fact in the differential and the difference case, the theory depends on the co-commutative Hopf algebras. 
When we  treat the \qsi algebras, the Hopf algebra 
$\mathcal{H}$ 
is not co-comutatives so that the 
Galois hull $\eL /\K$ that is an algebra extension in 
the non-commutative algebra of 
twisted  formal power series algebra, 
the dual algebra of $\mathcal{H}$. 
So even if we start from a (\,commutative\,) field,  extension $L/k$, the 
Galois hull can be non-commutative. See the Examples in sections 
\ref{10.4a}, \ref{10.4b} and \ref{10.4c}.
We also notice that when $L/k$ is a Picard-Vessiot  
extension fields in \qsi algebra, the 
Galois hull is commutative \cite{ume11}.
\par 
As the Galois hull is a non-commutative, 
if we limit ourselves to the category of commutative $L\n$-algebras $(Alg/L\n)$, 
we can not detect non-commutative nature of the \qsi 
field extension. So it is quite natural to extend the 
functors over the category of non-commutative algebras.    
\subsubsection{Infinitesimal deformation functor 
$\mathcal{F} \sb{L/k}$ for  a \qsi field extension.} 
 We pass to the task of defining the infinitesimal deformation functor $\mathcal{F}\sb {L/k}$ and the 
 infinitesimal automorphism functor $\infgal (L/k)$. 
 \par 
 We have the universal Taylor morphism 
 \begin{equation}\label{a4.10}
 \iota \sb{L\s}\colon L\s \to (
 L\n[[W\sb 1,\, W\sb 2, \, \cdots ,\,  W\sb d ]], 
 \{ 
 \frac{\partial}{\partial W\sb 1}, \, \frac{\partial}{\partial W\sb 2} ,\,  \cdots , \,  \frac{\partial}{\partial W\sb d} 
 \}
 )         
\end{equation}
 as in \eqref{m27.1}.
 So by \eqref{a4.10}, we have the canonical morphism 
\begin{equation}\label{a4.11}
 (F(\N , L\s )[[X]], \, \mathcal{D} ) \to (F(\N , \, L\n [[W]])[[X]],\,  \mathcal{D} ), 
 \end{equation}
 where in the target space 
 $$
 \mathcal{D} = 
 \{ \Sigma , 
 \frac{\partial}{\partial W\sb 1},\,  \frac{\partial}{\partial W\sb 2} , \, \cdots ,\,  \frac{\partial}{\partial W\sb d},  \,  
  \Theta ^ *    \}  
  $$
 by abuse of notation.  
 \par
 For an $L\n$-algebra $A$, the structure morphism 
 $ L\n \to A$ induces 
  the canonical morphism 
 \begin{equation}\label{a4.12}
 (F(\N ,\,  L\n [[W]])[[X]],\, \mathcal{D} ) \to (F(\N , \, A [[W]])[[X]],\,  \mathcal{D} ).
 \end{equation}
The composite of the $\mathcal{D}$-morphisms  \eqref{a4.11}  and \eqref{a4.12}  
gives us the  canonical morphism 
\begin{equation}\label{a4.13}
 (F(\N , \, L\s )[[X]],\,  \mathcal{D} ) \to 
 (F(\N , \, A [[W]])[[X]],\,  \mathcal{D} ).
\end{equation}
 The restriction of the morphism 
 \eqref{a4.13} to the $\D$-invariant sub-algebra $\eL$     
 gives us the canonical morphism 
\begin{equation}
\iota \colon (\eL,\,  \mathcal{D} ) \to  (F(\N , \, A [[W]]), \, \mathcal{D} ).
\end{equation}
 We can define the functors exactly as in paragraphs  \ref{9.25e} for the 
 differential case and \ref{9.25d} for the difference case.
 \begin{definition}\label{9.25f}
[Introductory definition]
 We define   the functor 
$$
\mathcal{F}\sb {L/k} \colon (Alg/ L\n) \to (Set)
$$
from the category $(Alg/L\n)$  of $L\n$-algebras to the category 
$(Set)$ of sets, by associating to an $L\n$-algebra $A$,  the set of  
 infinitesimal deformations of the canonical morphism \eqref{a4.13}.  
\par
Hence 
\begin{multline*}
\mathcal{F}\sb{L/k}(A)
= \{ f\colon (\eL ,\,  \D ) \to 
( F(\N ,\,  A [[ W\sb 1, \, W\sb 2, \, \cdots ,\,  W\sb d ]] )[[X]],\,  \D ) \, | \, 
 f \text{ is a  }
 \mathcal{D}\text{-morphism }\\
 \text{congruent to the canonical  morphism } \iota 
 \text{ modulo nilpotent elements} \\ 
 \text{such that } 
 f = \iota \text{ when restricted to the sub-algebra } \K 
\}.
 \end{multline*}
 \end{definition} 
The introductory definition \ref{9.25f} is exact, 
analogous to Definitions in \ref{9.25e} 
and \ref{9.25d}, and 
easy to understand.
Since as we explained  in \ref{9.25c}, we,
however,  
 have to consider also deformations over non-commutative algebras, the notation is confusing. 
\par 
We have to treat both the category of 
commutative algebras and that of non-commutative algebras. 
\begin{definition}
All the associative algebras that we consider 
are unitary and the morphisms between them are assumed to be unitary. For a commutative algebra $R$, we denote by $(CAlg/R )$  the category of  
associative commutative $R$-algebras.
We consider also the category of not necessarily commutative $R$-algebras. To be more precise 
we denote by $(NCalg/R)$ the category of associative $R$-algebras $A$ such that (\,the image of\,)  $R$ is in the center of $A$. When there is no danger of confusion the category of commutative algebras is denoted simply by $(Alg/R)$ 
\end{definition}
Let us come back to the \qsi field extension $L/k$. We can now give the infinitesimal deformation functors in an appropriate language.  
\begin{definition}
The functor $\mathcal{F}\sb {L/k}$ defined in \ref{9.25f} will be denoted by 
$\mathcal{CF}\sb {L/k}$. So we have 
$$
\mathcal{CF}\sb{L/k}\colon (CAlg/L\n) \to (Set). 
 $$
 We extend formally  the functor $\mathcal{CF}\sb {L/k}$ in \ref{9.25f}
from the category $(CAlg /L\n)$ to the category 
$(NCAlg/L\n)$.
Namely, we define the functor 
\begin{equation*}
\NCF \colon \NCA \to (Set)
\end{equation*}
by setting   
\begin{multline*}
\NCF (A)  = \{ 
\varphi \colon \eL \rightarrow F(\N, \, A[[W]])[[X]]  
\, | \,
\text{\it $\varphi$ is an injective $q$-SI $\sigma$-differential }\\
\text{\it morphism compatible with the derivations } 
\text{\it $\partial/\partial W1, \, 
\partial /\partial W\sb 2, \, \cdots ,\, \partial/\partial W\sb d 
$.}
 \}
\end{multline*}
for 
$A \in Ob\, (NCAlg/L\n)$. 
So it follows from the definition that the restriction of the functor
$\mathcal{NCF}\sb {L/k}$
 to the sub-category 
$(CAlg/L\n) $ is $\mathcal{CF}\sb {L/k}$ so that 
$$
\mathcal{NCF}\sb {L/k} \, |\,  {(CAlg/L\n)} = \mathcal{CF}\sb{L/k}.
$$
\end{definition}
In the examples, we consider \qsi structure, differential structure and difference structure of a given field extension $L/k$ and we study Galois groups with respect to the structures. 
So we have to clarify which structure is     
in question. 
For this reason, when we treat \qsi structure, we sometimes 
add suffix $\sigma \, \theta $ to indicate 
that we treat the \qsi  structure 
as in \ref{9.25a}.
For exsmple 
$\mathcal{NCF} \sb {\sigma \theta ^* L/k}$.
\subsubsection{Definition of quantum Galois group}The definition of the group functor $\infgal (L/k)$ is similar.
 \begin{definition}
 The Galois group in our Galois  theory is the group functor 
$$
\infgal (L/k) \colon (Alg/L\n ) \to (Grp)
$$
 defined 
 by 
 \begin{multline*}
 \infgal (L/k) (A) =
  \{
  \, 
 f \colon \eL \hat{\otimes} \sb {L\s} A[[W]] \to 
 \eL \hat{\otimes} \sb {L\s} A[[W]]
 \, | \, \\
 f 
 \text{
  is a } 
  \K \otimes \sb {L\s}A[[W]]   
  \text{-automorphism} 
  \text{ compatible with 
  $\mathcal{D}$, } \\
\text{ continuous with respect to
  the $W$-adic topology} \\
  \text{ and congruent to the identity modulo nilpotent elements 
  }
    \}
 \end{multline*}
 for an $L\n$-algebra $A$. See Definition 2.19 in \cite{mori09}. 
\end{definition}
\par
Then the group functor $\infgal (L/k)$ operates on the functor 
$\mathcal{F}\sb {L/k}$ in such a way that 
the operation $(\infgal (L/k), \mathcal{F}\sb {L/k} )$ is a 
principal homogeneous space. 
\par 
\section{The first example, the field extension $\com(t)/\com$}
\label{10.4a}
From now on, we assume $C =\com$. The arguments below work for an algebraic closed field $C$ of characteristic $0$. So $q$ is a non-zero complex number. 
\subsection{Analysis of the example}\label{10.10a}
Let $t$ be a variable over $\com$.  The field $\com(t) $ of rational functions has various structures: 
the differential field structure, the $q$-difference field structure and the \qsi structure that we are going to define. 
We are interested in the Galois group of the field extension $\com(t)/\com$ with respect to these structures. 
Let $\sigma \colon \com(t) \rightarrow \com(t)$ be the  $\com$-automorphism of the rational function field $\com(t)$ sending $t$ to $qt$. 
So $(\com(t), \,\sigma)$ is a difference field. We assume $q^{n}\neq 1 $ fore every positive integer $n$. We define a $\com$-linear map 
$\theta^{(1)}\colon \com(t) \rightarrow \com(t)$ by setting 
\[
 \theta^{(1)}(f(t)) := \frac{\sigma(f)-f}{\sigma(t) - t} = \frac{f(qt) -f(t)}{(q-1)t} \qquad \text{ for }f(t)\in \com(t).
\]
For an integer $n \geq 2$, we set
\[
 \theta^{(n)}:= \frac{1}{[n]\sb{q}!}\left(\theta^{(1)}\right)^{n}.
\]
It is convenient to define 
\[
 \theta^{(0)} = \Id\sb{\com(t)}. 
\]
It is well-known and easy to check that $(\com(t),\, \sigma, \, \theta^{*}) = (\com(t), \, \sigma, \, \{\theta^{(i)}\}\sb{i \in \N})$ is a \qsi algebra. 
\begin{lemma}\label{8.30a}
Galois group of the difference Picard-Vessiot extension $(\com(t), \, \sigma)/ (\com, \, \Id\sb{\com})$ is the multiplicative group $\G\sb{m\,\com}$. 
\end{lemma}
\begin{proof}
 Since $t$ satisfies the linear difference equation $\sigma(t) = qt$ over $\com$ and the field $C\sb{\com(t)}$ of constant of $\com(t)$ is $\com$, the extension
 $(\com(t),\, \sigma )/ (\com, \, \Id\sb{\com})$ is a difference Picard-Vessiot extension. 
The result follows from the definition of the Galois group. \end{proof}  
When $q\rightarrow 1$, the limit of the \qsi ring $(\com(t), \, \sigma, \, \theta^{*})$ is the differential algebra $(\com(t), \,  d/dt)$. 
We denote by $AF\sb {1\, k}$, the algebraic group of affine transformations of the affine line so that 
\[
 AF\sb{1\, \com } = \left\{ \left.
   \begin{bmatrix}
   a & b \\
   0 & 1
   \end{bmatrix} \right| 
   a, \, b \in \com ,\,  a \neq 0
   \right\} . 
\] Then
\[
 AF\sb{1\, \com }\simeq \G\sb{m\, \com} \ltimes \G\sb{a\,\com}, 
\]
where 
\begin{align*}
 \G\sb{m\,\com} &\simeq \left\{ \left.
 \begin{bmatrix}
 a & 0\\
 0 & 1
 \end{bmatrix} \in AF\sb{1\, \com } \right| a\in \com ^{*} \right\}, \vspace{1ex}\\
 \G\sb{a\,\com} &\simeq \left\{ \left.
 \begin{bmatrix}
 1 & b\\
 0 & 1
 \end{bmatrix} \in AF\sb{1\,\com} \right| b\in \com  \right\}.
\end{align*}
\begin{lemma}\label{8.30b}
The Galois group of differential Picard-Vessiot 
extension $(\com(t), \, d/dt)/\com$ is $\G\sb{a\,\com}$.
\end{lemma}
\begin{proof}
We consider the linear differential equation
\begin{equation}\label{8.10a}
Y^\prime = \left[ \begin{array}{cc}
0 & 1 \\
0 & 0
\end{array} \right] Y,
\end{equation}
where $Y$ is a $2 \times 2$-matrix with entries in a differential extension field of $\com$. 
Then $\com(t)/\com $ is the Picard-vessiot extension for \eqref{8.10a}, 
$$
Y =  \left[ \begin{array}{cc}
1 & t \\
0 & 1
\end{array} \right]
$$
being a fundamental solution of \eqref{8.10a}. 
The resalt is well-known and follows from, the definition of Galois group. 
\end{proof}
The extension $(\com(t), \, \sigma, \, \theta^{*})/\com$ is not a Picard-Vessiot extension so that we can not treat it in the framework of Picard-Vessiot theory. We can apply, however, Hopf Galois theory of Heiderich \cite{hei10}. 
\begin{proposition}\label{5.14a}
The Galois group of the extension $(\com(t),\,\sigma, \,  \theta^{*})/\com$ is isomorphic to the formal completion $\hat{\G}\sb{m\,\com}$ of the multiplicative group $\G\sb{m\,\com}$
\end{proposition}
Theory of Umemura \cite{ume96.2} and Heiderich\cite{hei10} single out only the Lie algebra. 
Proposition \ref{5.14a} should be understood in the following manner. 
We have two specializations of the \qsi field extension $(\com(t),\,  \sigma, \, \theta^{\ast})/\com$. 
\begin{enumerate}
\renewcommand{\labelenumi}{(\roman{enumi})}
\item $q \rightarrow 1$ giving the differential field extension $(\com(t), \, d/dt)/\com$. 
See \ref{9.19h}. 
\item  Forgetting $\theta ^*$, 
or equivalently specializing 
$$
\theta ^{(i)} \to 0 \qquad \text{\it for } 
i \ge 1,
$$
  we get   the difference field extension $(\com(t),\, \sigma)/\com$. See \ref{9.19c}.
\end{enumerate}
We can summarise the behavior of the 
Galois group under the specializations.
\begin{enumerate}
\renewcommand{\labelenumi}{(\arabic{enumi})}
\item Proposition \ref{5.14a} says that the Galois group of $(\com(t),\, \sigma, \, \theta^{*})/\com$ is the multiplicative group $\G\sb{m\,\com}$. This describes the 
Galois group at the  generic pooint. 
\item By Lemma \ref{8.30a}, Galois group of the specialization (i) is the multiplicative group $\G\sb{m\,\com}$.
\item Galois group of the specialization (ii) is the additive group $\G\sb{a\,\com}$ by Lemma \ref{8.30b}.
\end{enumerate}
\begin{proof}
Let us set $L = (\com(t),\, \sigma, \, \theta^{*})$ and $k = (\com, \, \sigma, \, \theta^{*})$. 
By definition of the universal Hopf  morphism 
\eqref{9.19a}, 
\[
 \iota \colon (L, \, \sigma , \, \,\,  \theta^* ) 
 \rightarrow 
 \left( F(\N,L\n)[[X]],\, \hat{\Sigma}, \, \hat{\Theta}^* \right),  \qquad \iota(t) = tQ + X \in F(\N,L\n)[[X]], 
, \]
where 
\[
 Q \in F(\N, L\n)
\]
is a function on $\N$ taking values in $\com \subset L\n$ such that 
\[
 Q(n) = q^{n} \qquad \text{ for } n \in \N.
\]
We denote the function Q by 
\[
 Q = \begin{bmatrix}
      0 & 1 & 2 & \cdots \\
      1 & q & q^{2} & \cdots
     \end{bmatrix}
\]
according to the convention in \cite{mori09}.
We take the derivation $\partial/\partial t \in \text{Der}(L\n/k\n)$ as a basis of the 1-dimensional $L\n$-vector space $\text{Der}(L\n/k\n)$ of $k\n$-derivations of $L\n$.
So $(\partial/\partial t) (\iota(t)) = Q$ is an element of the Galois hull $\eL$. Therefore 
\[
 \eL = L\n<X,\, Q>\sb{alg},  
\]
which is the $L\n$-sub-algebra of 
$F(\N, \, L\s )[[X]]$ generated by $X$ and $Q$. 
Since $QX = q XQ$, the Galois hull $\eL$ is a non-commutative $L\n$-algebra. 
Now we consider the universal Taylor expansion 
\[
 (L\n, \partial/\partial t) \rightarrow L\n[[W]]
\]
and consequently we have 
\begin{equation}\label{10.3d}
 \iota \colon \eL \rightarrow F(\N, L\n)[[X]] \rightarrow F(\N,L\n[[W]])[[X]]. 
\end{equation}
We study infinitesimal deformation of $\iota$ in \eqref{10.3d}. 
Let A be a commutative $L\n$-algebra
\[
 \varphi \colon \eL \rightarrow F(\N,\, A[[W]])[[X]]
\]
be an infinitesimal deformation of 
$$
\iota :\eL \to F(\N , \, A[[W]])[[X]],  
$$
for a commutative $L\n$-algebra $A$ so that 
$$
\varphi \in \mathcal{F} \sb {L/k}(A).
$$

\begin{sublemma}\label{5.15}
(1) There exists a nilpotent element $n \in A$ such that $\varphi(Q) = (1+ n)Q$.
(2) $\varphi(X) = X$. 
\end{sublemma}
Sublemma proves Proposition\ref{5.14a}. 
\end{proof}
\begin{proof}[Proof of Sublemma.] The elements 
$X, Q \in \eL$ satisfy the following equation. 
\begin{align*}
 &\frac{\partial X}{\partial W} = \frac{\partial Q}{\partial W} = 0,\\
 &\hat{\Sigma}(X) = qX,\qquad\hat{\Sigma}(Q) = qQ, \\
 &\hat{\Theta}^{(1)}(X) = 1, \qquad \hat{\Theta}^{(i)}(X) = 0 \qquad \text{ for } i\geq 2,\\
 &\hat{\Theta}^{(i)}(Q) = 0 \,\,\text{ for } i \geq 1. 
\end{align*}
So $\varphi(X), \, \varphi(Q)$ satisfy the same equations as above, which shows 
\begin{align*}
\varphi(X) &= X + fQ \in F(\N,A[[W]])[[X]],\\
\varphi(Q) &= eQ \in F(\N,A[[W]])[[X]],\\ 
\end{align*}
where $f,e \in A$. Since $\varphi$ is an infinitesimal deformation of $\iota$, $f$ and $1-e$ are nilpotent elements in $A$. 
It remains to show $f = 0$. In fact, it follows from the equation
\[
 QX = qXQ
\]
that 
\[
 \varphi(Q)\varphi(X) = q\varphi(X)\varphi(Q)
\]
or
\[
 eQ(X + fQ) = q(X + fQ)eQ.
\]
So we have 
\[
 eQfQ = qfQeQ
\]
and so 
\[
 efQ^{2} = qfeQ^{2}.
\]
Therefore
\[
 ef = qfe.
\]
Since $e$ is a unit, $e-1$  being nilpotent in $A$,
\[
f - qf = 0,
\]
so that 
\[
(1-q)f = 0.
\]
As $1-q$ is a non-zero complex number, $f = 0$ as we wanted. 
\end{proof}
We have shown that the functor
\[
\mathcal{F}\sb{L/k} \colon (Alg/L\n) \rightarrow (Set)
\]
is a principal homogeneous space of the group functor $\hat{G}\sb{m\,\com}$. 
See paragraph \ref{9.28a} Origin of the group structure as well as paragaraph \ref{10.5a} below. 
\par
During the proof of Proposition \ref{5.14a}, we
have proved the following 
\begin{proposition}\label{10.17a}
The Galois hull $\eL$ coincides with the sub-algebra 
$$
L\n< X, \, Q>\sb {alg}
$$
of $F(\N, \, L\n)[[X]]$ generated by $L\n , \, X$ and $Q$. The commutation relation of $X$ and $Q$ is 
$$
QX=qXQ. 
$$ 
In particular, if $q\not= 1$, then the Galois hull is non-commutative.
\end{proposition}

We are ready to describe the non-commutative deformations.
\begin{lemma}\label{8.30c}
If $q \neq 1$,  we have 
\[
 \NCF (A)= \left\{ \,  
 e \in A,\,  f \in A,\, | \,  fe=qef \text{ and } 1-e,\,  f \text{ are nilpotent \,}
 \right\}
 \]
 for every $A\in  Ob (NCAlg/L\n)$.
\end{lemma}
\begin{proof}
Since $q \not= 1$,  
it follows from the argument of the proof of Sublemma \ref{5.15} that if we take 
$$
\varphi \in 
\mathcal{NCF}\sb {L/k\sigma \theta^*} \qquad \text{\it for } A \in
Ob \, \NCA , 
$$ 
then $\varphi(X) = X + f$ and $\varphi(Q) = eQ$, $f,e \in A$. \par
Since $\varphi$ is an infinitesimal deformation of $\iota$, $f$ and $1-e$ are nilpotent. \par
It follows from $QX = qXQ$ that 
\[
eQ(X+f) = q(X + f)eQ
\]
so $ef = qfe$. 
\end{proof}
We are going to see in \ref{10.5a} that
theoretically, 
 we can identify 
\begin{equation}
 \NCF(A) = \left\{ \begin{array}{cc}
 \left.  \begin{bmatrix}
 e & f\\
 0 & 1
 \end{bmatrix} \right| 
 e \in A, f \in A, fe=qef \text{ and } 1-e, f \text{ are nilpotent}
 \end{array}
 \right\} . 
\end{equation}
\begin{cor}
[Corollary to the proof of Lemma \ref{8.30c}]
When $q = 1$ that is the case excluded in our general study, we consider the \qsi differential field 
$$
(\com(t),\, \Id , \, \theta^{\ast})
$$ 
as in \ref{9.19b}. So $\theta^{\ast}$ is the iterative derivation; 
\begin{align*}
\theta^{(0)} &= \Id,  \\
 \theta^{(i)} &= \dfrac{1}{i!}\dfrac{d^i}{dt^i}\qquad \text{for } i\geq 1 .
\end{align*}
Then we have 
\begin{align}
 & \eL\sb{\sigma  \, \theta ^*} \simeq \eL\sb{d/dt}, 
 \label{9.26a} \\ 
& \mathcal{NCF}
\sb{(( \com(t),\, \Id, \,\theta^{\ast})/\com)
}(A)
= \{ f \in A \, | \, f \text{\it \/ is a nilpotent element }
\}  \label{9.26b}
\end{align}
 for $ A \in  Ob \NCA$. 
\end{cor}
\begin{proof}
In fact if $q = 1$, then 
\begin{equation*}
\left[\begin{array}{ccccc}
0 & 1 & 2 & 3 & \cdots \\
1 & 1 & 1 & 1 & \cdots
\end{array} \right] = 1 \in \com.
\end{equation*}
So $\eL\sb{q,\sigma}$ is generated by $X$ over $\K$. Therefore $\eL\sb{\sigma \, \theta^*} \simeq \eL\sb{d/dt}$. 
Since $Q = 1 \in \K, \, \varphi(Q) = Q$ for an inifinitesimal deformation 
 $$
 \varphi \in \mathcal{NCF}\sb{\sigma , \, \theta ^*}(A)
 $$ 
 and we get \eqref{9.26b}. 
\end{proof}
\subsubsection{Quantum group enters}\label{9.28b}
To understand Lemma \ref{8.30c}, it is convenient to introduce a quantum group. 
\begin{definition}
We work in the category $(NCAlg/\mathbb{C})$. Let $A$ be a not necessarily commutative $\mathbb{C}$-algebra. 
We say that  two subsets $S,\,  T$ of $A$ are mutually commutative if for every $s\in T,\, 
 t \in T$, we have $[s,t]=st-ts = 0$.
\end{definition}
For $A \in  Ob \NCA$, we set
\[
H\sb{q}(A) =\{ 
 \begin{bmatrix}
e & f \\
0 & 1 
\end{bmatrix} \, | \, 
e,\,  f \in A, \text{\it \/ $e$ is invertible in $A$, } qef=fe\}
. \]
\begin{lemma} For two matrices
\[
Z\sb{1} = \begin{bmatrix}
e\sb{1} &f\sb{1} \\
 0 & 1 
\end{bmatrix},\quad  
Z\sb{2} = \begin{bmatrix}
e\sb{2} & f\sb{2}\\
0       & 1 
\end{bmatrix} \in H\sb{q}(A), 
\]
if $\{e\sb{1},\, f\sb{1}\}$ and $\{e\sb{2},\, 
f\sb{2}\}$ are mutually commutative, then the product matrix 
\[
Z\sb{1} Z\sb{2} \in H\sb{q}(A).
\]
\end{lemma}
\begin{proof}
Since 
\[
Z\sb{1} Z\sb{2} = \begin{bmatrix}
e\sb{1}e\sb{2} & e\sb{1}f\sb{2} + f\sb{1}\\
 0 & 1
\end{bmatrix}, 
\]
we have to prove 
\[
qe\sb{1}e\sb{2}(e\sb{1}f\sb{2}+f\sb{1}) = (e\sb{1}f\sb{2} + f\sb{1})e\sb{1}e\sb{2}
. \]
This follows from the mutual commutativity of $\{e\sb{1},\,  f\sb{1}\}$,and   $\{e\sb{2},\,  f\sb{2}\}$,  
and the conditions $qe\sb{1}f\sb{1} = f\sb{1}e\sb{1},\, qe\sb{2}f\sb{2} = f\sb{2}e\sb{2}$. 
\end{proof}
\begin{lemma}
For a matrix
\[
Z = \begin{bmatrix}
e & f\\
0 & 1
\end{bmatrix} \in H\sb{q}(A),
\]
if we set
\[
\tilde{Z} = \begin{bmatrix}
e^{-1} & -e^{-1}f \\
 0  & 1
\end{bmatrix} \in M\sb{2},
\]
then 
\[
\tilde{Z} \in H\sb{q^{-1}}(A) \text{ and } 
\tilde{Z}Z = Z\tilde{Z} = I\sb{2}. 
 \]
\end{lemma}
\begin{proof}
We can check it by a simple calculation. 
\end{proof}
\begin{remark}
If $q^{2}\neq 1$, for $f \neq 0$ $\tilde{Z} \not\in H\sb{q}(A)$. In fact let us set
\[
\tilde{Z} = \begin{bmatrix}
\tilde{e} & \tilde{f}\\
0 & 1
\end{bmatrix}
\]
so that $\tilde{e}=e^{-1}, \tilde{f} = -fe^{-1}$. 
Then $\tilde{e}\tilde{f} = e^{-1}(-fe^{-1}) = -qfe^{-2}$ and $\tilde{f}\tilde{e}= -fe^{-1}e^{-1} = -fe^{-2}$. 
So $q\tilde{e}\tilde{f} = -q^{2}fe^{-2}\neq -fe^{-2} = \tilde{f}\tilde{e}$, if $q^{2} \neq 1$, then and $\tilde{Z} \in H\sb{q^{-1}}(A)$. 
\end{remark}
For an algebra $R$, a sub-algebra $S$ of $R$ and a sub-set $T \subset R$, we denote by 
$S< T>\sb{alg}$ the sub-algebra of $R$ generated over $S$ by $T$.
\begin{lemma}
Let $u$ and $v$ be symbols over $\com$.
We have shown that we find a $\mathbb{C}$-Hopf algebra 
\begin{equation} \label{hopfa1}
\gH\sb q = \mathbb{C}<u,\, u^{-1}, \, v> \sb{alg}/(uv -q^{-1}vu)
\end{equation}
as an algebra so that
$$
uu^{-1} = u^{-1} u = 1 , \qquad u^{-1} v = qvu^{-1}.
$$
Definition of the algebra $\gH\sb q$ determines the multiplication
\[
m \colon \gH \sb q\otimes\sb{\mathbb{C}}\gH\sb q \to \gH\sb q, 
\]
the unit 
\[
\eta \colon \mathbb{C} \to \gH\sb q, 
\]
that is the composition of natural
 morphisms $$\com \rightarrow \com<u, \, u^{-1}, \, v>\sb{alg}$$ and $$\com<u, \, u^{-1}, \, v>\sb{alg} \rightarrow \com<u, \, u^{-1}, \, v>\sb{alg}/(uv-q^{-1}vu)=\gH\sb q.$$ 
 The product of matrices gives 
the co-multiplication
\[
\Delta  \colon \gH\sb q \to \gH\sb q\otimes\sb{\mathbb{C}}\gH\sb q, 
\] 
that is a $\com$-algebra morphism 
defined by
\[
\Delta (u)=u\otimes u, \qquad \Delta (u^{-1})=
u^{-1}\otimes u^{-1}, 
\qquad \Delta (v)=u\otimes v + v\otimes 1,
\]
for the generators $u,\, u^{-1},  \, v$ of the algebra ${\gH\sb q}$, 
the co-unit is a $\com$-algebra morphism 
\[
\epsilon \colon \gH\sb q \to \mathbb{C}, \qquad \epsilon(u)=\epsilon (u^{-1})=1, \, \epsilon(v)=0 
 \]
 for the generators $u, \, u^{-1}, \, v$ of the algebra 
${\gH\sb q}$.
The antipode 
$$
i\colon {\gH\sb q} \to {\gH\sb q}
$$
is a $\com$-anti-algebra morphism given by 
$$
i(u)=u^{-1}, \qquad i(u^{-1})= u,  \qquad (v) = -u^{-1}v.  .
$$
\end{lemma}
Let us set
$$
\gH\sb {q\, L\n} := \gH\sb q\otimes \sb{\com} L\n
$$
so that $\gH\sb{q \, L\n}$ is an $L\n$-Hopf algebra. We 
notice that for an $L\n$algebra $A$ 
$$
\begin{array}{rl}
\gH\sb{q\, L\n}(A) &:= \Hom\sb{L\n\text{-algebra}}
(\gH\sb{q \,L\n}, \, A)  \\ 
&= 
 \{ 
\begin{bmatrix}
\begin{array}{cc}
e & f \\
0 & 1
\end{array} 
\end{bmatrix} \, | \, e, \, f \in A, \, ef=q^{-1} fe,  \text{ $e$ is invertible}
 \}.
\end{array}
$$

\begin{remark}
 We know by general theory that 
 the antipode $i\colon H \to H$ that is a linear map 
 making a few diagrams commutative, is necessarily an anti-endo-morphism of the algebra $H$ so
 that   
 $$
 i (ab) = i(b) i(a)  
 \text{\it\/  for all elements $a, \, b \in H$ and \/ }
i(1) =1.
 $$
\end{remark} 
See Manin \cite{man88}, section 1,\,2. 
\par 
The Hopf algebra $\gH\sb q$ is a $q$-deformation of the affine algebraic group $AF\sb {1\, \com}$ of affine transformations of the affine line. 
\par 
Anyhow, we notice that the quantum group 
appears in this very simple example showing that 
quantum groups are indispensable for a Galois theory of \qsi field extensions.

\subsection{Observations on the Galois structures 
of the field extension $\com (t)/ \com $}\label{10.5a}
Let us now examine that the group structure in \ref{9.28a} 
arising from the variation of initial conditions coincides with the 
quantum group structure defined in \ref{9.28b}.
\par 
To see this, we have to clearly understand 
the initial condition of 
a formal series 
$$
f(W, \, X) = \sum\sb {i= 0}^ \infty X^i a\sb i (W) \in 
F(\N , \, A[[W]] )[[X]]
$$
so that the coefficients $a\sb i$'s, which are elements of 
$F(\N ,\, A[[W]])$,  
are 
functions on $\N$ taking values in the formal power series ring $A[[W]]$.
The initial condition of $f(W,\, X)$ is the value of 
the function $f(0, \, 0)=a\sb{0} (0) \in F(\N, \, A)$ at $n=0$ which we may denote by 
$$
f(0, \, 0)|\sb {n=0}\in A. 
$$
\par
For $A \in Ob(NCAlg/L\n)$, we take an 
infinitesimal deformation 
$
\varphi \in \NCF{} (A)
$
so that the morphism         
$
\varphi \colon \eL \to F(\N,\, A[[W]])[[X]]
$
is determined by the image $\varphi (y)$ of 
$y \in L \subset \eL$, the \qsi firld 
$L$ being a sub-algebra of $\eL$ 
by the universal Hopf morphism. 
It follows from Lemma \ref{8.30c} that there exist $e,\, f \in A$
such that $qef=fe$, 
 the elements $e-1, \, f$ are nilpotent and such that   
\begin{equation}\label{9.28d}
\varphi (t) = \left( e(t + W) +f\right) Q +X.
\end{equation}
The above equality \eqref{9.28d} says 
that in the level of the initial condition, 
the dynamical system 
$ 
y \mapsto \varphi (y)
$
looks 
as
\begin{equation}\label{9.28e}
t \mapsto \text{ the initial condition of }\varphi (y) =
et + f.
\end{equation}
The composition of two mutually commutative transformations  
of the form \eqref{9.28e} is nothing but the multiplication of $2\times 2$ matrices. 
Therefore the quantum group structure is the same as in 
the group structure in 
\ref{9.28a}.
\par
The Hopf algebra $H$ in \ref{928b} defines a 
functor 
$$
\hat{\gH}\sb{q \, L\n}\colon (NCAlg/L\n )\to (Set)
$$
such that 
\begin{multline*}
\hat{\gH}\sb{q\, L\n}(A) =\{ \psi :\gH\sb q\otimes \sb\com L\n \to A \, |\, \psi \text{ is a $L\n$-algebra morphism }\\
\text{such that  }\psi (u) -1, \, 
\psi (v) \text{ are nilpotent} \} 
\end{multline*}
for $A \in (NCAlg/L\n)$.
In other words $\hat{\gH}\sb{q\, L\n}$ is the formal completion of the quantum group $\gH\sb q \otimes \sb \com L\n=
\gH\sb{q\, L\n}$.  
We can summarize our results in the following form. 
\begin{theorem}
The formal quantum group $\hat{\gH}\sb{q\, L\n}$ operates on the functor $\NCF$ in such a way that there exists a functorial isomorphisms 
\[
\NCF \simeq \hat{\gH}\sb{q\, L\n}. 
\]
The restriction of the functor $\NCF$ on the subcategory $(CAlg/L\n)$ gives the functorial isomorphism 
\[
\NCF |\sb{(CAlg/L\n)} \simeq \hat{\mathbb{G}}\sb{m\, L\n}.
\]
Or equivalently, 
\begin{enumerate}
\renewcommand{\labelenumi}{(\arabic{enumi})}
\item The infinitesimal Galois group of the \qsi extension $(\com (t),\,  \sigma , \,  \theta^{\ast} ) / \com$ on the category 
$(NCAlg/L\n )$ 
 of not necessarily commutative $L\n$-algebra is isomorphic to the formal 
 quantum group $\hat{\gH}\sb{q \, L\n}$. 
\item The infinitesimal Galois group of the \qsi extension $(\com (t), \sigma , \theta^{\ast} ) / \com$ on the category $(Alg / L\n )$ of  commutative $L\n$-algebras 
 is isomorphic to the formal group 
$\hat{\G}\sb{m}$.  
\end{enumerate}
\end{theorem}
\par 
The operation of formal quantum group requires a 
precision. 
\begin{remark}\label{10.24a}
We should be careful about the 
operation of quantum formal group. To be more precise, for $\varphi \in 
\mathcal{F}\sb{L/k}(A)$ and
$
 \psi   \in \hat{\gH}\sb{q\, L\n}(A)
$ 
so that 
we have 
$$
\varphi (t) = (e(t + W) +f)Q +X \in F(\N , \,A[[W]])[[X]]
$$
with $e, \, f\in A$ and 
 we imagine the
matrix 
$$
\left[ \begin{array}{cc}
\psi(u) & \psi(v) \\
 0 & 1 
\end{array} \right] \in M\sb 2 (A)
$$
corresponding to $\psi$. 
If 
the sub-sets of the algebra $A$,  $\{ \psi(u), \, \psi(v) \}$ and $\{ e, \, f \}$ are commutative,
the product 
$$
\psi \cdot \varphi = \varpi   \in \mathcal{F}\sb{L/k}(A)
$$
is defined  
to be
$$
\varpi (t) = (\psi(u) e (t+ W)+ \psi(u)f + \psi (v) )Q    + X \in F(\N , \, A[[W]])[[X]].
$$ 
\end{remark}
\section{The second example, the \qsi field extension $(\com(t,\, t^{\alpha}) , \,\sigma,\, \theta^{\ast})/\com$ }\label{10.4b}
\subsection{Commutative deformations}\label{1027a}
As in the previous section, let $t$ be a variable  over $\com$ 
and we assume that the complex number $q$ is not a root of unity   if we do not mention other assumptions on $q$. Sometimes we write the condition that 
$q$ is not a root of unity, simply   to 
recall it. 
We work under the condition  that $\alpha $ is an irrational complex number so that $t$ and $t^\alpha$ are algebraically independent over $\com$. 
Therefore the field $\com (t, \, t^\alpha )$
is isomorphic to the rational function field of two variables over $\com$. We denote 
by 
$\sigma$ the $\com$-automorphism of the field $\com (t, \, t^\alpha )$ such that 
$$
 \sigma (t) = qt \text{ and }
 \sigma (t^\alpha ) =q^\alpha t^\alpha .
$$
Let us set $\theta^{(0)} := \Id\sb{\com(t,t^{\alpha})}$, the map
\[
 \theta^{(1)}: = \frac{\sigma - \Id}{(q-1)t} \colon \com(t,t^{\alpha}) \rightarrow \com(t,t^{\alpha})
\]
and
\[
 \theta^{(n)} = \frac{1}{[n]\sb{q}!} \left( \theta^{(1)} \right)^{n}\qquad  \text{for}\qquad n = 2, 3, \cdots .
\]
So the  the $\theta^{(i)}$'s are  $\com$-linear operators on $\com(t,\, t^{\alpha})$ and 
\[
 L \colon = (\com(t,\, t^{\alpha}), \,\sigma,\,  \theta^{\ast})
\]
is a \qsi field. The restriction of $\sigma$ and $\theta^{\ast}$ to the subfield $\com$ are trivial. 
We denote the \qsi field  extension $L/\com$ by $L/k$. We denote $t^{\alpha}$ by $y$ so that the abstract field $\com(t,t^{\alpha}) = \com(t,\, y)$ is isomorphic to the rational function field of $2$ variables over$\com$. We take the derivations $\partial/ \partial t$ and $\partial / \partial y$ as a basis of the $L\n$-vector space $\mathrm{Der}(L\n/k\n)$ of $k\n$-derivations of $L\n$. Hence $L\s= (L\n,\{\partial/\partial t, \partial/\partial y \})$ as in \cite{ume11}. \par
Let us list the fundamental equations.
\begin{align}
&\sigma(t) = qt, \qquad \sigma(y) = q^{\alpha}y, \label{5.24a}\\
&\theta^{(1)}(t) = 1, \qquad 
                        \theta^{(1)}(y) = [\alpha]\sb{q} \frac{y}{t}. \label{5.24d}
\end{align}
We are  going to determine the Galois group $\mathbf{NC}\infgal(L/k)$. Before we start, we notice
that since 
by Proposition \ref{10.17a}, the 
Gaois hull of the extension $\com (t), \, \sigma , \, \theta ^* )/\com $ is not a commutative algebra and 
sine $ ( \com (t) $ is a sub-firld of $\com ( t, \, t^\alpha )$,  the Galois hull of the \qsi field extension 
$(\com (t, \, t^\alpha ), \, \sigma .\,\theta ^*)/\com $ is not a  commutative algebra either. 
Consequently  the \qsi  field extension  
$\com (t, \, t^\alpha )/\com $ is not a Picard-Vessiot extension (See  \cite{ume11}).
So we have to go beyond the general theory of Heiderich \cite{hei10}, Umemura \cite{ume11} 
 for the definition  of the Galois group $\mathbf{NC}\infgal(L/k)$. \par
It follows from general definition that the universal Hopf morphism 
\[
\iota \colon L \rightarrow F(\N, L\n)[[X]]
\]
is given by 
\[
 \iota(a) = \sum\sb{n=0}^{\infty}X^{n}u[\theta^{(n)}(a)] \in F(\N,L\n)[[X]]
\]
for $a \in L$. Here for $b \in L$, we denote by $u[b]$ the element
\[
u[b] = \left[ \begin{array}{cccc}
0 & 1 & 2 & \cdots\\
b & \sigma(b) & \sigma^{2}(b) & \cdots
\end{array} \right]
\in F(\N,L\n)
. \]
It follows from the definition above of the universal Hopf morphism $\iota$, 
\[
 \iota(y) = \sum\sb{n=0}^{\infty}X^{n}{\binom{\alpha}{n}}\sb{q}t^{-n}Q^{\alpha - n}y, 
\]
where we use the following notations. For a complex number $\beta \in \alpha + \Z$ , 
\[
 [\beta]\sb{q} = \frac{q^{\beta} - 1}{q-1}
\]
and
\[
 \binom{\alpha}{n}\sb{q} = \frac{[\alpha]\sb{q}[\alpha -1]\sb{q} \cdots [\alpha -n + 1]\sb{q}}{[n]\sb{q}!}. 
\]
\[
Q = \left[ \begin{array}{cccc}
0 & 1 & 2 & \cdots\\
1 & q & q^{2} & \cdots
\end{array} \right]
\qquad \text{and}\qquad 
Q^{\alpha} = \left[ \begin{array}{cccc}
0 & 1 & 2 & \cdots\\
1 & q^{\alpha} & q^{2\alpha} & \cdots
\end{array} \right]
. \]
We set
\[
 Y\sb{0} := \sum\sb{n=0}^{\infty}X^{n}\binom{\alpha}{n}\sb{q} t^{-n}Q^{\alpha -n}
\]
so that 
\begin{equation}\label{5.21b}
 \iota(y)= Y\sb{0}y \qquad \text{\it in } F(\N,L\n)[[X]]. 
\end{equation}
Considering $k\n$-derivations $\partial/\partial t, \partial /\partial y $ in $L\n$ and therefore in $F(\N,L\n)$ or in $F(\N,L\n)[[X]]$, 
we generate the Galois hull $\eL$ by $\iota(L)$ and $L\n$ so that $\eL \subset F(\N,L\n)[[X]]$ is invariant under $\Sigma$, 
the $\Theta^{(i)}$ and $\{ \partial/\partial t, \partial/ \partial y \}$. 
We may thus consider
\[
 \eL \hookrightarrow F(\N, L\s)[[X]]
. \]
By the universal Taylor morphism 
\[
 L\s = (L\n, \{\partial / \partial t, \partial/ \partial y \}) \rightarrow L\s[[W\sb{1},W\sb{2}]]
. \]
We identify $\eL$ by the canonical morphism
\[
 \iota \colon \eL \rightarrow F(\N, L\s)[[X]] \rightarrow F(\N, L\n[[W\sb{1},W\sb{2}]])[[X]]. 
\]
We study  first  the deformations of $\iota$ on the category $(CAlg/L\n)$ of commutative $L\n$-algebras and then generalize the  argument to the category $(\NCA)$ of not necessarily commutative $L\n$-algebras. \par
For a commutative $L\n$-algebra $A$, let $\varphi \colon \eL \to F(\N,A[[W\sb{1},W\sb{2}]])[[X]]$ 
be an infinitesimal deformation of the canonical morphism $\iota \colon \eL \rightarrow F(\N, L\n[[W\sb{1},W\sb{2}]])[[X]]$ 
so that both $\iota$ and $\varphi$ are compatible with operators $\{ \Sigma, \Theta^{\ast}, \partial /\partial W\sb{1},\partial / \partial W\sb{2} \}$. 
\begin{lemma}
The infinitesimal deformation $\varphi$ is determined by the images $\varphi(Y\sb{0}),\, \varphi(Q)$ and $\varphi(X)$. 
\end{lemma}
\begin{proof}
The Galois hull $\eL/\K$ is generated over $\K= L\s$ by $\iota(t) = tQ + X$ and $\iota(y) = Y\sb{0}y$ with operators $\Theta^{\ast}, \, \Sigma$ and $\partial / \partial t, \, \partial /\partial y$ along with localizations. This proves the Lemma. 
\end{proof}
Let us set $Z\sb{0} := \varphi(Y\sb{0}) \in F(\N, A[[W\sb{1},W\sb{2}]])[[X]]$ and expand it into a formal power series in X:
\[
 Z\sb{0} = \sum\sb{n = 0}^{\infty}X^{n}a\sb{n}, \,\, \text{ with } a\sb{n}\in F(\N,A[[W\sb{1},W\sb{2}]]) 
\qquad  \text{for every } n \in \N 
. \]
It follows from \eqref{5.24a} and \eqref{5.21b}
\[
 \Sigma(Z\sb{0}) = q^{\alpha}Z\sb{0}
\]
so that 
\begin{equation}\label{6.2a}
 \sum\sb{n=0}^{\infty}X^{n}q^{n}\Sigma(a\sb{n}) = q^{\alpha} \sum\sb{n=0}^{\infty} X^{n}a\sb{n}. 
\end{equation}
Comparing the 
coefficient of the $X^{n}$'s in \eqref{6.2a} we get
\[
 \Sigma(a\sb{n}) = q^{\alpha -n}a\sb{n} \,\, \text{ for } n \in \N.
\]
So $a\sb{n} = b\sb{n}Q^{\alpha - n}$ with $b\sb{n} \in A[[W\sb{1},W\sb{2}]]$ for $n \in \N$. Namely we have 
\begin{equation}\label{5.24g}
 Z\sb{0} = \sum\sb{n=0}^{\infty}X^{n}b\sb{n}Q^{\alpha -n} \qquad \text{with}\qquad  b\sb{n} \in F(\N, A[[W\sb{1},W\sb{2}]]) .
\end{equation}
It follows from \eqref{5.24d}
\[
 \sigma(y) - y = \theta^{(1)}(y)(q-1)t
\]
and so by \eqref{5.24a}
\[
(q^{\alpha} - 1)y = \theta^{(1)}(y)(q-1)t.
\]
Applying the canonical morphism $\iota$ and the deformation $\varphi$, we get 
\begin{equation}
(q^{\alpha}-1)Y\sb{0}= \Theta^{(1)} (Y\sb{0})(q-1)(tQ+X) \label{5.24e}
\end{equation}
as well as 
\begin{equation}
(q^{\alpha}-1)Z\sb{0}= 
\Theta (Z\sb{0})(q-1)(teQ+X) . \label{5.24f}
\end{equation}
Substituting \eqref{5.24g} into \eqref{5.24f}, we get a recurrence relation among the $b\sb{m}$'s; 
\[
 b\sb{m+1} = \frac{[\alpha - m]\sb{q}}{[m+1]\sb{q}(e(t + W\sb{1}))}b\sb{m}. 
\]
Hence 
\begin{equation}\label{11.16a}
 b\sb{m} = \binom{\alpha}{m}\sb{q}(e(t+W\sb{1}))^{-m}b\sb{0} \qquad \text{for every } m \in \N,
\end{equation}
where $b\sb{0} \in A[[W\sb{1}, W\sb{2}]]$ and every coefficient of the power series $b\sb{0}-1$ are nilpotent. \\
Since 
\[
 \frac{\partial Y\sb{0}}{\partial y} 
= \frac{\partial}{\partial W\sb 2}\left( \sum\sb{n=0}^{\infty} X^{n} \binom{\alpha}{n}\sb{q}(t+W\sb{1})^{-n} Q^{\alpha -n} \right) = 0,
\]
we must have 
$$
0 =\varphi ( \frac{\partial Y\sb{0}}{\partial y})
=\frac{\partial \varphi(Y\sb{0})}{\partial W\sb 2}=
\frac{\partial Z\sb{0}}{\partial W\sb 2}
$$
 and consequently 
\[
 \frac{\partial b\sb{0}}{\partial W\sb{2}} = 0
\]
so that 
\[
 b\sb{0} \in A[[W\sb{1}]]
. \]
 by \eqref{5.24g}.
Therefore, we have determined the image 
\[
Z\sb 0 = \varphi(Y\sb{0}) = \sum\sb{n=0}^{\infty} X^{n} \binom{\alpha}{n}\sb{q}(e(t+W\sb{1}))^{-n} Q^{\alpha -n}b\sb{0}
\]
by \eqref{11.16a}, 
where all the coefficients of the power series $b\sb{0} -1$ are nilpotent. 
\subsection{The functor 
$NC\mathcal{F}\sb{ L/k}$ 
of infinitesimal deformations 
restricted on the category $(CAlg/L\n)$ of commutative $L\n$-algebras }
We can summarize what we have proved as follows. 
\begin{proposition}\label{5.24i}
There exists a functorial inclusion on the category 
$(CAlg/L\n)$ of commutative $L\n$-algebras 
\begin{multline}
\mathcal{NCF}\sb{L/k}(A) 
 \hookrightarrow \hat{G}\sb{II}(A) :=\{ (e,\,  b(W\sb{1})) 
   \in A \times A[[W\sb 1]]\, |\, \\
       \text{\it all the coefficients of } b(W\sb 1) 
       \text{ and } e-1 \text{ are nilpotent } \} 
      \end{multline}
for every commutative $L\n$-algebra $A$. 
\end{proposition}
\begin{proof}
In fact, we send a deformation $\varphi \in \mathcal{NCF}\sb{L/k}(A)$ to $(e,\,  b\sb{0}(W)) \in A\times A[[W\sb{1}]] $ that is an element  of the sub-set $\hat{G}\sb{II}(A)$ of $A\times A[[W\sb{1}]]$.
\end{proof}
\begin{conjecture}\label{5.24j}
If $q$ is not a root of unity, 
the inclusion in Proposition \ref{5.24i} is the equality. 
\end{conjecture}
The set 
\[
 \hat{G}\sb{II}(A) = \left\{ (e, \, b(W\sb{1})) \in A\times A[[W\sb{1}]] \left| \begin{array}{l}
\text{\it $e-1\in A$  and all the coefficients of}\\
\text{\it  the power series $b(W\sb{1})-1$ are nilpotent}
\end{array} \right. \right\}
\]
has a natural group structure functorial in $A \in Ob(Alg/L\n)$. \par
Namely for two elements $(e\sb{1}, \, b\sb{1}(W\sb{1})),\, (e\sb{2},b\sb{2}(W\sb{1})) \in \hat{G}\sb{II}(A)$, the product is given by 
\[
(e\sb{1}, \, b\sb{1}(W\sb{1})) \times (e\sb{2}, \, b\sb{2}(W\sb{2})) 
= (e\sb{1}e\sb{2},\, \sb{1}(e\sb{2}W\sb{1} + e\sb{2}-1)b\sb{2}(W\sb{1})). 
\]
the unit being $(1,\, 1) \in \hat{G}\sb{II}(A)$ and the inverse 
\[
(e, \, b(W\sb{1}))^{-1} = (e^{-1},\, b(e^{-1}W\sb{1} + e^{-1} - 1)^{-1})
. \]
is an element of $\hat{G}\sb {II}(A)$. 
\begin{proposition}\label{5.24k}
If Conjecture \ref{5.24j} is true, we have isomorphism 
\[
NC\infgal(L/k)\, |\,  \sb{(CAlg/:\n)} \simeq \hat{G}\sb{II}. 
\]
of group functors. 
\end{proposition}
\begin{remark}\label{5.24l}
We explain a background of Conjecture \ref{5.24j}. 
\end{remark}
\begin{lemma}\label{5.24n} The Galois hull $\eL$ is a localization of the following ring
\[
 L\s\left[Q, X, \frac{1}{tQ + X}\right]\left[\frac{\partial^{l}}{\partial t^{l}}Y\sb{0}\right]\sb{l\in\N}. 
\]
\end{lemma}
\begin{proof} Since $\iota(t) = t Q + X$, as we have seen in the first exampe,  
\[
 L\s[Q, X]\left[\frac{\partial^{l}}{\partial t^{l}}Y\sb{0}\right]\sb{l\in\N} \subset \eL.
\]
We show that the ring
\[
L\s[Q, X]\left[\frac{\partial^{l}}{\partial t^{l}}Y\sb{0}\right]\sb{l\in\N}
\]
is closed under the operations $\Sigma, \Theta^{(i)}, \partial/\partial t$ and $\partial/\partial y$ of $F(\N, \, L\s)[[X]]$. Evidently the ring is closed under the last two operators. Since the operators $\Sigma$ and $\partial^{n} / \partial t^{n}$ operating on $F(\N, L\s)[[X]]$ mutually commute, it follows from \eqref{5.24a}
\[
 \Sigma\left( \frac{\partial^{n} Y\sb{0}}{\partial t^{n}} \right) = \frac{\partial^{n}}{\partial t^{n}} \Sigma(Y\sb{0}) = \frac{\partial^{n}}{\partial t^{n}}(q^{\alpha}Y\sb{0}) = q^{\alpha}\frac{\partial^{n} Y\sb{0}}{\partial t^{n}}
. \]
So the ring is closed under $\Sigma$. 
Similarly since the operators $\Theta^{(1)}$ and $\partial^{n} / \partial t^{n}$ mutually commute on $F(\N, L\s)[[X]]$,
\begin{align*}
\Theta^{(1)}\left( \frac{\partial^{n} Y\sb{0}}{\partial t^{n}} \right) 
&=\frac{\partial^{n}}{\partial t^{n}} \Theta^{(1)}(Y\sb{0})\\
&=\frac{1}{y}\frac{\partial^{n}}{\partial t^{n}} \Theta^{(1)}(Y\sb{0}y)\\
&=\frac{1}{y}\frac{\partial^{n}}{\partial t^{n}} \Theta^{(1)}(\iota(y))\\
&=\frac{1}{y}\frac{\partial^{n}}{\partial t^{n}} \iota(\theta^{(1)}(y))\\
&=\frac{1}{y}\frac{\partial^{n}}{\partial t^{n}} \iota \left( \frac{\sigma(y) - y}{(q-1)t}\right)\\
&=\frac{1}{y}\frac{\partial^{n}}{\partial t^{n}} \left( \frac{q^{\alpha}Y\sb{0}y - Y\sb{0}y}{(q-1)(tQ+X)}\right)\\
&=\frac{1}{y}\frac{\partial^{n}}{\partial t^{n}} \left( \frac{q^{\alpha}Y\sb{0} - Y\sb{0}}{(q-1)(tQ+X)}\right),
\end{align*}
which is an element of the ring. 
\end{proof}
Conjecture \ref{5.24j} arises from   {\it experience that if $q$ ls not a root of unity, 
 it is very hard to find a non-trivial algebraic relations among
 the partial derivatives} 
 $$
\frac{\partial^{n} Y\sb{0}}
{\partial t^{n}}\qquad
\text{\it  for } n \in \N 
$$
{ \it over $L\s$
   so that we could guess that there would be none.} 
\par 
In fact, assume that we could  prove our guess. Let $\varphi \colon \eL \rightarrow 
F(\N, A[[W\sb{1}, W\sb{2}]])[[X]]$ be an infinitesimal deformation of $\iota$. So as we have seen
\[
 Z\sb{0} = \varphi(Y\sb{0}) = \sum\sb{n=0}^{\infty} X^{n} \binom{\alpha}{n}\sb{q}(et)^{-n}Q^{\alpha-n}b(W\sb{1})
\]
with $b(W\sb{1}) \in A[[W\sb{1}]]$. 
There would be no constraints among the partial derivatives 
$\partial^{n} b(W\sb{1})/\partial W\sb{1}^{n},\, n \in \N$ and hence we could choose any power series $b(W\sb{1}) \in A[[W\sb{1}]]$. 
\subsection{The functor $\NCF$ of non-commutative deformations}\label{11.23a}
We study the functor $\NCF(A)$ of non-commutative deformations 
\[
\NCF \colon (\NCA) \rightarrow (Set).
\]
For a not necessarily commutative $L\n$-algebra $A \in Ob(\NCA)$, let
\begin{equation}\label{9.13a}
 \varphi \colon \eL \rightarrow F(\N,
 \, A[[W\sb{1}, W\sb{2}]])[[X]] 
\end{equation}
be an infinitesimal deformation of the canonical morphism
\[
\iota \colon \eL \rightarrow F(\N, A[[W\sb{1},W\sb{2}]])[[X]]
. \]
Both $t$ and $y$ are elements of the field $\com(t,t^{\alpha} )= \com(t,y))$ so that $[t,y] = ty -yt = 0$. So for the deformation $\varphi \in \NCF(A)$ we must have 
\begin{equation}\label{6.4a}
[\varphi(t),\varphi(y)] = \varphi(t)\varphi(y) - \varphi(y)\varphi(t) = 0.
\end{equation}
When we consider the non-commutative deformations, the commutativity \eqref{6.4a} gives a constraint for the deformation. 
To see this we need a Lemma. 
\begin{lemma}\label{6.2g}
For every $l\in \N$, we have
\[
 q^{l}\binom{\alpha}{l}\sb{q} + \binom{\alpha}{l-1}\sb{q} = \binom{\alpha}{l} \sb q  + 
 q^{\alpha -l +1}\binom{\alpha}{l-1}\sb{q}. 
\]
\end{lemma}
\begin{proof} 
This follows from the definition of q-binomial coefficient. 
\end{proof}
\begin{lemma}\label{9.13b}
Let $A$ be a not necessarily commutative $L\n$-algebra in $Ob\, (NCAlg/L\n)$. Let $e, \, f \in A$ such that $e-1$ and $f$ are nilpotent. We set $\mathcal{A} : = (e(t+W\sb{1}) + f)Q + X$ and for a power series $b(W\sb{1})\in A[[W\sb{1}]]$, we also set
\[
 \mathcal{Z} := \sum\sb{n=0}^{\infty} X^{n} 
  \binom{\alpha}{n}\sb{q}(e(t+W\sb{1})+f)^{-n}
 Q^{\alpha -n}b(W\sb{1})
\]
so that $\mathcal{A}$ and $\mathcal{Z}$ are elements of $F(\N,A[[W\sb{1}]])[[X]]$. 
The following conditions are equivalent. 
\begin{enumerate}
\renewcommand{\labelenumi}{(\arabic{enumi})}
\item $[\mathcal{A},\, \mathcal{Z}] := \mathcal{A}\mathcal{Z} - \mathcal{Z}\mathcal{A} = 0$. 
\item $[e(t + W\sb{1}) + f, \, b(W\sb{1})] = 0$. 
\end{enumerate}
\end{lemma}
\begin{proof}
We formulate condition $(1)$ in terms of coefficients of the power series in $X$. Assume condition $(1)$ holds so that we have 
\begin{equation}\label{6.2d}
\begin{split}
((e(t+W\sb{1} + f))Q + X)\left( \sum\sb{n= 0}^{\infty}X^{n} \binom{\alpha}{n}\sb{q}(e(f+W\sb{1})+ f)^{-n}Q^{\alpha -n }b(W\sb{1}) \right)\\
=\left( \sum\sb{n= 0}^{\infty}X^{n} \binom{\alpha}{n}\sb{q}(e(f+W\sb{1})+ f)^{-n}Q^{\alpha -n }b(W\sb{1}) \right)((e(t+W\sb{1} + f))Q + X). 
\end{split}
\end{equation}
Comparing degree $l$ terms in $X$ of \eqref{6.2d}, we fined condition $(1)$ is equivalent to 
\begin{equation}\label{6.2e}
\begin{split}
&q^{l}\binom{\alpha}{l}\sb{q}(e(t+ W\sb{1})+ f)^{-l+1}Q^{\alpha -l+1}b(W\sb{1}) \\
&\hspace*{5em}
+ \binom{\alpha}{l-1}\sb{q}(e(t+W\sb{1})+f)^{-l+1}Q^{\alpha -l+1} b(W\sb{1})\\
&=
\binom{\alpha}{l}\sb{q}(e(t+ W\sb{1})+ f)^{-l}b(W\sb{1}) (e(t+ W\sb{1})+ f)Q^{\alpha -l +1}\\
&\hspace*{5em}
+ \binom{\alpha}{l-1}\sb{q}q^{\alpha -l+1}(e(t+W\sb{1})+f)^{-l+1}Q^{\alpha -l+1} b(W\sb{1}).
\end{split}
\end{equation}
That is equivalent to 
\begin{equation}\label{6.2f}
\begin{split}
&\hspace*{2em}q^{l}\binom{\alpha}{l}\sb{q}(e(t+ W\sb{1})+ f)b(W\sb{1}) \\
&\hspace*{5em}+ \binom{\alpha}{l-1}\sb{q}(e(t+W\sb{1})+f)b(W\sb{1})\hspace{5em}\\
&=\binom{\alpha}{l}\sb{q}b(W\sb{1}) (e(t+ W\sb{1})+ f)\\
&\hspace*{5em}+ \binom{\alpha}{l-1}\sb{q}q^{\alpha -l+1}(e(t+W\sb{1})+f)b(W\sb{1})
\end{split}
\end{equation}
for every $l \in \N$. Con dition \eqref{6.2f} for $l = 0$ is condition (2). Hence condition (1) implies condition (2). 
Conversely condition (1) follows from (2) in view of \eqref{6.2f} and Lemma \ref{6.2g}. 
\end{proof}
Now let us come back to the infinitesimal deformation \eqref{9.13a}
of the canonical morphism $\iota$.
The argument in Section \ref{10.4a} allows us to determine the restriction $\varphi$ on the subalgebra generated by $\iota(t) = tQ + X$
 over $L\s$ 
 invariant under the $\Theta^{(i)}$'s, 
 $\Sigma$ and $\{\partial/ \partial t, \partial/ \partial y \}$ in $F(\N, L\s)[[X]]$. So there exist $e, \, f \in A$ such that $ef= q^{-1}fe,\, e-1, \, f$ are nilpotent and such that
\[
 \varphi(Q) = eQ \,\,\text{ and } \varphi(X) = fX + Q,
\]
that are equations in $F(\N,A[[W\sb{1},W\sb{2}]])[[X]]$. 
In particular 
\[
 \varphi(t) = \varphi(tQ+X) = (et+f)Q+ X = (e(t+W\sb{1})+ f)Q + X,
\]
where we naturally identify rings
\[
F(\N,L\s)[[X]] \rightarrow F(\N, L\n[[W\sb{1},W\sb{2}]])[[X]] \rightarrow F(\N, A[[W\sb{1},W\sb{2}]])[[X]]
\]
through the canonical maps. \par
Then the argument in 
the commutative case allows us to show  that there exists a power series $b\sb{0}(W\sb{1}) \in A[[W\sb 1]]$ such that 
\[
 \varphi(Y\sb 0) = \sum\sb{n=0}^{\infty} X^{n} \binom{\alpha}{n}\sb{q}(e(t+ W\sb{1})+ f)^{-n}Q^{\alpha - n}b\sb{0}(W\sb{1}).
\]
such that all the voefficients of the powerseris $b\sb 0(W\sb 1)$ are nilpotent.
As we  deal with the not necessarily commutative algebra $A$, the commutation relation in $L$ gives a constraint. Namely since $\iota(y) = yY\sb{0}$ and $t y = y t$ in $L$ so that $\iota(t)\iota(y) = \iota(y)\iota(t)$, we get $\iota(t)(yY\sb{0}) = (yY\sb{0})\iota(t)$ in $\eL$ and 
$
\varphi (tQ + X ) \varphi (Y\sb 0) =\varphi (Y\sb 0) \varphi ( tQ + X)
$
.
So we consequently have 
\begin{equation}\label{6.8a}
\mathcal{A}\mathcal{Z}\sb 0 = \mathcal{Z}\sb{0}\mathcal{A} \qquad \text{ in } F(\N, A[[W\sb{1}, \, W\sb{2}]])[[X]], 
\end{equation}
setting 
\[
 \mathcal{A} : = e((t+W\sb{1}) + f)Q + X, \qquad \mathcal{Z}\sb 0 := \sum\sb{n=0}^{\infty} X^{n}  \binom{\alpha}{n}\sb{q}(e(t+W\sb{1})+f)Q^{\alpha -n}b\sb 0 (W\sb{1}).
\]
\begin{lemma}]\label{10.23a}
We have 
\[ [e(t + W\sb{1}) + f,\, b\sb 0 (W\sb{1})] = 0.
\]
\end{lemma}
\begin{proof}
This follows from \eqref{6.8a} and Lemma 
\ref{9.13b}. 
\end{proof}
\begin{definition}\label{10.19a}
We define a functor 
\[
 QG\sb{II\, q} \colon (\NCA) \rightarrow (Set)
\]
by putting 
\begin{multline*}
 QG\sb{II \, q}(A) 
  = \{ ( \begin{bmatrix}
 e & f\\
 0 & 1
 \end{bmatrix}, \, b(W\sb{1}) 
 ) \in M\sb {2} (A)\times A[[W\sb 1]]  
\,  |\, 
 e, \, f \in A,\,  ef = q^{-1}fe, \\
 e \text{\it  is invertible in } A , \, 
  b(W\sb 1 ) \in A[[W\sb 1]],\, [e(t+W\sb{1}) + f, \, \, b(W\sb{1})] = 0
 \, 
\} 
\end{multline*}
for $A\in Ob\,(NCAlg/l\n).$ 
\par The functor  $QG\sb {II\, q}$ is almost 
a quantum group.   
We also need the formal completion $\widehat{QG}\sb {II\, q}$  of the quantum  group functor $QG\sb{II\, q}$ so that 
$$
\widehat{QG}\sb {II\, q}:(NCAlg/L\n) \to (Set)
$$
is given by 
\begin{multline*}
\widehat{ QG}\sb{II \, q}(A) 
= \{ ( \begin{bmatrix}
 e & f \\
 0 & 1
 \end{bmatrix}, \, b(W\sb{1})) \in QG\sb {II\, q}(A) \,  \\ 
 |\,  
 $e-1,\,  f$ \text{ and all the coefficients of } 
 b(W\sb 1 ) \text { are nilpotent}
 \} 
\end{multline*}
for $A\in  Ob (NCAlg/L\n).$
\end{definition}
Using Definition \ref{10.19a}, we
have shown the following 
\begin{proposition}\label{6.8i}
There exists a functorial inclusion 
\[
 \NCF(A) \hookrightarrow \widehat{QG}\sb{II \, q}(A)
\]
 sending $\varphi \NCF(A)$ to 
\[
  ( \begin{bmatrix}
 e & f\\
 0 & 1
 \end{bmatrix},\, b\sb 0 (W\sb{1}) )\in \widehat{QG}\sb{II \, q}(A). 
\]
\end{proposition}
We show that $\widehat{QG}\sb{II \, q}$ is a quantum fodmal group over $L\n$. In fact, we take two elements 
\[
(G,\, \xi(W\sb{1})) =
  ( \begin{bmatrix}
 e & f\\
 0 & 1
 \end{bmatrix},\,  \xi (W\sb{1}) ),\qquad  
(H,\, \eta(W\sb{1})) =
  ( \begin{bmatrix}
 g & h\\
 0 & 1
 \end{bmatrix}, \,  \eta (W\sb{1}) )
\]
of $\widehat{QG}\sb{II \, q}(A)$ so that $e,\,f, g,\, h \in A$ so that 
\[
 ef = q^{-1}fe, \qquad  gh= q^{-1}hg,
\]
the elements $ e$ and $g$ are invertible  and such that
\begin{equation}\label{6.6e}
[e(t + W\sb{1}) + f, \, \xi(W\sb{1})] = 0, \qquad [g(t+W\sb{1}) + h, \, \eta(W\sb{1})] = 0.
\end{equation}
When the following two subsets  of the ring  $A$ 
\begin{enumerate}
\renewcommand{\labelenumi}{(\arabic{enumi})}
\item
$\{e, f \} \cup $ 
 ( the subset of all the coefficients of the power series $\xi(W\sb{1}))$, 
 \item 
$\{g,\,  h \}\cup $( 
the subset of 
all the coefficients of the 
power series $\eta(W\sb{1}))$, 
\end{enumerate}
  are mutually commutative, 
we define the product of $(G, \, \xi(W\sb{1}))$ and $(H, \, \eta(W\sb{1}))$ by
\[
(G, \xi(W\sb{1}))\star(H, \, \eta(W\sb{1})) = (GH,\, \xi(gW\sb{1} + (g-1)t + h)\eta(W\sb{1})).
\]

\begin{lemma}
The product $(GH,\, \xi(gW\sb{1} + (g-1)t + h)\eta(W\sb{1}))$ is indeed an element of $\widehat{QG}\sb{II \, q}(A)$. 
\end{lemma}
\begin{proof}
First of all, we notice that
the constant term  $(g-1)t$ of 
 the linear polynomial in $W\sb 1$ 
\begin{equation}\label{10.22a}
gW\sb{1} + (g-1)t + h
\end{equation}
is nilpotent so that we can substitute 
\eqref{10.22a}
into the power series $\xi (W\sb 1 )$. Therefore 
 $$
 \xi(gW\sb{1} + (g-1)t + h)\eta(W\sb{1})
 $$
is a well-determined element of the power series ring 
$A[[W\sb 1]]$.
We have seen in Section \ref{10.4a} that if 
  $\{ e, \, f \}$
  and $\{g, \, h \} $
  are mutually commutative, then  
the product $GH$ of matrices $G, \, H \in \gH\sb{q \, L\n}(A)$ 
 is in 
 $\gH\sb{q \, L\n}(A)$. 
 Since 
$GH = \begin{bmatrix}
 eg & eh + f\\
 0 & 1
\end{bmatrix}$, it remains  to show 
\begin{equation}\label{6.6a}
[eg(t+W\sb{1})+ eh +f, \, \xi(gW\sb{1}+ (g-1)t + h )\eta(W\sb{1})] = 0. 
\end{equation}
The proof of \eqref{6.6a} is done in several steps. \par
First we show 
\begin{equation}\label{6.6x}
[\xi(gW\sb{1}+ (g-1)t + h ), \,  \eta(W\sb{1})]= 0. 
\end{equation}
This follows, in fact, from the mutual commutativity 
of the subsets (1) and (2) above, and the second equation of \eqref{6.6e}. \par
Second we show 
\begin{equation}\label{6.6g}
[eg(t+W\sb{1})+ eh +f,\,  \xi(gW\sb{1}+ (g-1)t + h )] = 0.
\end{equation}
To this end, we notice
\begin{equation}\label{6.6c}
eg(t+W\sb{1})+ eh +f = e(gW\sb{1}+ (g-1)t + h) + et + f.
\end{equation}
So we have to show 
\begin{equation}
[e(gW\sb{1}+ (g-1)t + h) + et + f, \, 
\xi(gW\sb{1}+ (g-1)t + h )]=0.
\end{equation}
This follows from 
the first equation of \eqref{6.6e}
and the mutual commutativity of the subsets (1) and (2).. \par
We prove third 
\begin{equation}\label{6.8d}
[eg(t+W\sb{1})+ eh +f, \, \eta(W\sb{1})]=0 .
\end{equation}
In fact, this a consequence of the second equation of 
\eqref{6.6e} and the mutual commutativity of the subsets (1) and (2).
Equality  \eqref{6.6a} is a consequence of \eqref{6.6x}, \eqref{6.6g} and \eqref{6.8d}. 
\end{proof}
One can check associativity for the multiplication by a direct calculation. The unit element is given by 
\[
(I\sb{2}, 1) \in \widehat{QG}\sb{II \, q}(L\n). 
\]
The antipode is given by the formula below. 
For an element 
$$
(G, \, b(W\sb 1) ) =
( \begin{bmatrix}
e & f \\ 
0 & 1
\end{bmatrix}, \, b(W\sb 1) )
\in \widehat{QG}\sb{II\, q}(A),               
$$
 we set 
$$
(G, \, b(W\sb 1 )  )^{-1}: = ( 
\begin{bmatrix}
e^{-1} & -e^{-1}f\\
0& 1
\end{bmatrix}, \,
b(e^{-1}W\sb 1) + (e^{-1}t -e^{-1}f) ^{-!}b(W\sb 1)
 ) \in \widehat{QG}\sb{II\, q^{-1}}(A), 
$$
then we have 
$$
(G, \, b(W\sb 1) )^{-1}\star (G, \, b(W\sb 1) ) =
(G, \, b(W\sb 1) )\star (G, \, b(W\sb 1) )^{-1} = (I\sb 2 , \, 1 ).
$$
\begin{conjecture}\label{6.8j}
If $q$ is not a root of unity, 
the injection in Proposition \ref{6.8i} is an equality for every $A \in Ob(CAlg/L\n)$. 
\end{conjecture}
\begin{proposition}
Conjecture \ref{6.8j} implies Conjecture \ref{5.24j}
\end{proposition}
\begin{proof}
Let us assume Conjecture \ref{6.8j}. Take an element $(e,\,  \xi(W\sb{1}))\in \hat{G}\sb{II}(A)$ for $A \in Ob\, (Alg/L\n)$. 
Since $A$ is commutative, the commutation relation 
in Lemma \ref{10.23a}
imposes no condition on $\xi(W\sb{1})$, $(e, \, \xi(W\sb{1}))\in \widehat{QG}\sb{II \, q}(A)$. Conjecture \ref{6.8j} says that 
if $q$ is not a root of unity,  
 $(e, \, \xi(W\sb{1}))$ arise from an infinitesimal deforma
\[
 \iota \colon \eL \rightarrow F(\N,\, A[[W\sb{1},W\sb{2}]])[[X]]. 
\]
\end{proof}
Conjecture \ref{6.8j} says that we can identify the functor $\NCF$ with the quantum formal group $\widehat{QG}\sb{I\, qI}$. 
 To be more precise, 
the argument in the first Example studied in \ref{10.4a}, allows us to
define a formal $\com$-Hopf algebra $\hat{\mathfrak{I}}\sb q$ 
and hence 
$$
\hat{\mathfrak{I}}\sb{q\, L\n}:= \widehat{\mathfrak{I}}\sb q \hat{\otimes} \sb \com L\n, 
$$
which  is a functor on the category $(NCAlg/L\n )$\
so that we have a functoril isomorphism 
$$
\hat{\gI}\sb {q\, L\n}(A) \simeq \widehat{QG}\sb{II\, q}(A)
\qquad \text{for every } L\n\text{-algebra }A 
\in Ob\, (NCAlg/L\n).
$$ 
\subsection{Summary on the Galois structures of the field  extension $\com(t,\, t^{\alpha})/\com$ }
Let us summarize our results on the  $(\com(t, \, t^{\alpha})/\com)$. 
\begin{enumerate}
\renewcommand{\labelenumi}{(\arabic{enumi})}
\item Difference field extension $(\com(t,\, t^{\alpha}),\, \sigma)/\com$. This is a Picard-Vessiot extension with Galois group $\G\sb{m\, \com} \times \G\sb{m\,\com}$. 
\item Differential field extension $(\com(t,\, t^{\alpha}),\, d/dt)/\com$. This is not a Picard-Vessiot extension. The Galois group 
$$
\infgal(L/k)\colon (CAlg/L\n ) \to (Grp)
$$
 is isomorphic to $\hat{\G}\sb{m\, L\n} \times \hat{\G}\sb{a \,L\n}$, 
 where $\hat{\G}\sb{m L\n}$ and 
 $\hat{\G}\sb {a L\n}$ are formal completion of the multiplicative group and the additive group. 
 So as group functors on the category 
 $(CAlg/L\n)$,  we have  
 $$ 
 \hat{\G }\sb{ m L\n}(A) =\{ b\in A \, | \, b-1 \text{\it\/ is nilpotent} \}
 $$
 and 
 $$
 \hat{\G }\sb{ a L\n}(A) =\{ b\in A \, | \, b \text{\it \/ is nilpotent} \}
  $$   
  for a commutative $L\n$-algebra $A$.
  
\item Commutative deformation of \qsi extension $(\com(t,\, t^{\alpha}),\, \sigma, \, \theta^{\ast})/\com$. 
If $q$ is not a root of unity, 
$\infgal(L/k)$ is an infinite dimensional formal group such that we have 
\[
 0 \rightarrow \widehat{A[[W\sb{1}]]^{\ast}} \rightarrow \infgal(L/k)(A) \rightarrow \hat{\G}\sb{m}(A) \rightarrow 0,
\]
where $\widehat{A[[W\sb{1}]]^{\ast}}$ denotes the multiplicative group 
\[
 \left\{ a \in A[[W\sb{1}]]\, \left| \, \text{\itshape all the coefisients of power series $a-1$ are nilpotent}
\right. \right\} .
\]
modulo Conjecture \ref{6.8j}. 
\item Non-commutative Galois group. If $q$ is not a root of unity, the infinitesimal deformation functor is isomorphic to a quantum  formal group:
\[
\NCF \simeq \widehat{QG}\sb{II \, q}, 
\]
modulo Conjecture \ref{6.8j}. 
\par
We should be careful about the group structure. 
Quantum formal group structure in $\widehat{QG}\sb {q L\n}$ coincides with the group structure defined from the initial conditions as in Remark \ref{10.24a}.  
So we might say that non-commutative Galois group is the quantum formal group $\widehat{QG}\sb{II \, q}$. 
\item 
Let us assume $q$ is not a root of unity.
If we have a $q$-difference field extension 
$(L, \, \sigma )/(k, \, \sigma )$ such that $t \in  
L$ with $\sigma (t) = qt$, then we 
can define the  operator 
$\theta ^{(1)}\colon L \to L$ 
by setting 
$$\theta ^{(1)}(a) := \frac{\sigma (a) -a}{qt -t}.$$
We also assume the field $k$ is $\theta ^{(1)}$ invariant. 
Defining the operator 
 $
\theta ^{(n)}\colon L \to L
 $
 by 
 
\begin{align}
\theta^{(0)} & = \Id \\ 
 \theta ^{(n)} & = \frac{1}{[n]\sb q !}(\theta ^{(1)})^n
 \end{align}
 for every positive integer $n$ 
 so that we have a \qsi field extension 
 $(L, \, \sigma , \, \theta ^*)/(k, \, \sigma , \, \theta ^*)$.
 \par       
Here arises a natural question
 of comparing the Galois groups of 
 the difference field extension 
 $ (L, \, \sigma )/(k, \, \sigma  )$
  and \qsi field extension 
  $(L, \, \sigma , \, \theta^* )/(k, \, \sigma , \, \theta ^* )$. 
 \par
 As the \qsi field extension is constructed from the  
 difference field extension in a more or less trivial way,  
 one might  imagine that they coincide or they are not much different. 
 \par
 This contradicts Conjecture \ref{6.8j}. 
 Let us take our example $\com (t , \, t^\alpha )/\com $. 
 Assume Conjecture \ref{6.8j} is true.  Then the 
 Galois group for the \qsi extension is 
 $\widehat{QG}\sb {II q L\n}$ that is infinite dimensional,
 whereas 
 the the Galois group is of 
 the difference field extenswion is of 
  dimension $2$.
\end{enumerate}
\section{The third example, the field extension 
$\com (t, \, \log\, t )/ \com $}\label{10.4c}
We assume $q$ is a complex number not equal to $0$. 
Let us study the field extension $L/k:=\com(t, \, \log\, t)/\com$ from various view points as in Sections \ref{10.4a} and \ref{10.4b}. 
\subsection{$q$-difference field extension $\com(t,\, \log\, t)/\com$. }\label{6.8k}
We consider $q$-difference operator $\sigma\colon L \to L$ such that $\sigma$ is the $\com$-automorphism of the field $L$ satisfying 
\begin{equation}\label{6.7a}
 \sigma(t) = qt \qquad  \text{ and } \qquad \sigma(\log \, t) = \log \,t + \log\, q . 
\end{equation}
\par
It follows from \eqref{6.7a} that
if $q$ is not a root of unity, then the field of constants of the difference field 
$( \com (t, \, \log\, t), \, \sigma )$ is $\com $ and hence  
 $(\com(t, \, \log\, t),\sigma)/ \com$ is a Picard-Vessiot extension with Galois group $\G\sb{m\,\com} \times \G\sb{a \, \com}$
\subsection{Differential field extension $(\com(t, \, \log\, t), d/dt)/\com$. }
As we have
\[
 \frac{dt}{dt} = 1 \qquad \text{ and }\qquad  \frac{d\log\, t}{dt} = \frac{1}{t}.
\]
So both differential field extensions $\com(t,\, \log\, t)/\com(t)$ and $\com(t)/\com$ are Picard-Vessiot extensions with Galois group $\G\sb{a\, \com}$. The  differential extension $\com(t, \, \log\, t)/\com$ is not, however, a Picard-Vessiot extension. Therefore we need general differential Galois theory \cite{ume96.2} to speak of the Galois group of the differential field extension $\com(t, \, \log\, t)/ \com$. \\
The universal Taylor morphism 
\[
 \iota \colon L \rightarrow L\n[[X]]
\]
sends
\begin{align}
\iota(t) &= t+X,\label{6.7c}\\
\iota(\log\, t) &= \log\, t +\sum\sb{n=0}^{\infty}(-1)^{n+1}\frac{1}{n}\left( \frac{X}{t} \right)^{n} \in L\n[[X]] .\label{6.7d}
\end{align}
Writing $\log\, t$ by $y$, we take $\partial/\partial t,\, \partial/\partial y$ as a basis of $L\n=\com(t,y)\n$-vector space $\mathrm{Der}(L\n/k\n)$ of $k\n$-derivations of $L\n$. It follows from \eqref{6.7c}, \eqref{6.7d} that 
\[
 \eL = \text{ \it a localization of the algebra  }L\s [t+X, \, \sum\sb{n=1}^{\infty}(-1)^{n+1}\frac{1}{n}\left( \frac{X}{t} \right)^{n} ] \subset L\s[[X]].
\]
We argue as in  \ref{10.5a} and Section
 \ref{10.4b}.
 For a commutative algebra $A \in Ob(CAlg/L\n)$ and $\varphi \in \mathcal{F}\sb{L/k}(A)$, there exist 
nilpotent elements 
$a, b \in A $ such that $a,\, b$ such that
\begin{align*}
 \varphi(t+X) &= t+ W\sb{1} + X + a, \\
 \varphi ( \sum\sb{n=1}^{\infty}(-1)^{n+1}\frac{1}{n}\left( \frac{X}{t+W\sb{1}}\right)^{n} ) 
 &= \sum\sb{n=1}^{\infty}(-1)^{n+1}\frac{1}{n}\left( \frac{X}{t+W\sb{1}+ a}\right)^{n} + b
.\end{align*}
Therefore 
we arrived at the dynamical system 
\begin{equation}\label{10.31a}
\left\{ \begin{array}{l} t, \\ y,      \end{array} \right. \mapsto 
 \left\{\begin{array} {l}
\varphi  (t) = t + X +W\sb 1 +a , 
 \\
\varphi (y) = y+  \sum\sb{n=1}^{\infty}(-1)^{n+1}\frac{1}{n}\left( \frac{X}{t+W\sb{1}+ a}\right)^{n} + b .
 \end{array} \right. 
\end{equation}
In terms of initial conditions, dynamical system 
\eqref{10.31a} reads 
$$
\left\{ \begin{array}{l} t, \\ y,      \end{array} \right. \mapsto 
 \left\{\begin{array} {l}t +a ,  \\
y+  b ,
 \end{array}    \right. 
$$
where $a, \, b $ are nilpotent elements of 
of $A$.  
So we conclude 
$$
\infgal (L/k)(A) = \hat{\G}\sb a (A) \times \hat{\G}\sb a (A)
$$
for every commutative $L\n$-algebra $A$.
Consequently we get 
\[
 \infgal(L/k) \simeq  (\hat{\G}\sb{a, \com} \times \hat{\G}\sb{a, \com} ) \otimes\sb {\com} L\n . 
\]
\subsection{\qsi field extension $(\com(t,\, \log\, t),\, \sigma, \, \theta^{\ast})/\com$ }
$\sigma \colon \com(t, \, \log\, t) \rightarrow \com(t, \, \log\, t)$ is the automorphism in Subsection \ref{6.8k}. 
We set $\theta^{(0)} = \Id\sb{\com(t, \, \log\, t)}$ and 
\[
\theta^{(1)} = \frac{\sigma - \Id\sb{\com(t, \, \log\, t)}}{(q-1)t}
\]
so that $\theta^{(1)} \colon \com(t, \, \log\, t) \rightarrow \com(t, \, \log\, t)$ is a $\com$-linear map. We farther introduce 
\[
 \theta^{(i)} := \frac{1}{[i]\sb{q}!}(\theta^{(1)})^{i}\colon \com(t, \, \log\, t) \rightarrow \com(t, \, \log\, t)
\]
that is a $\com$-linear map for $i = 1,2,3\cdots$. 
Hence if we denote the set $\{ \theta^{(i)} \}\sb{i \in \N}$ by $\theta^{\ast}$, then $(\com(t,\, \log\, t),\, \sigma, \,\theta^{\ast})$ is a \qsi ring. \par 
The universal Hopf morphism 
\[
\iota \colon \com(t, \, \log\, t) \rightarrow F(\N, L\n)[[X]]
\]
sends, by Proposition \ref{a4.1},  $t$ and $y$ respectively to
\begin{align*}
 \iota(t) &= tQ + X\\
 \iota(y) &= y + (\log\, q) N + \frac{\log\, q }{q-1} \sum\sb{n=1}^{\infty} X^{n}(-1)^{n+1}\frac{1}{[n]\sb{q}}(tQ)^{-n}
\end{align*}
that we identify with 
\begin{align*}
& = y + W\sb{2} + \, (\log\,  q)\,N + \frac{\log\, q }{q-1} \sum\sb{n=1}^{\infty} X^{n}(-1)^{n+1}\frac{1}{[n]\sb{q}}(t+W\sb{1})^{-n}Q^{-n}, \,\, \in F(\N, A[[W\sb{1},W\sb{2}]])[[X]],
\end{align*}
where we set 
\[
N := \begin{bmatrix}
0&1&2& \cdots \\
0&1&2& \cdots
\end{bmatrix} \in F(\N, \Z)
. \]
We identify further $t+X$ with $t + W\sb{1}+X \in L\n[[W\sb{1},W\sb{2}]][[X]]$ and 
\[
 \sum\sb{n=1}^{\infty}(-1)^{n+1}\frac{1}{[n]\sb q}\left( \frac{X}{t}\right)^{n}
\]
with 
\[
 \sum\sb{n=1}^{\infty}(-1)^{n+1}\frac{1}{[n]\sb q}\left( \frac{X}{t+W\sb{1}}\right)^{n} \in L\n[[W\sb{1},W\sb{2}]][[X]]. 
\]
\subsubsection{Commutative deformations $\mathcal{F}\sb{L/k}$ for $(\com(t,\, \log\, t),\,  \sigma, \, \theta^{\ast})/\com$}
Now the argument of Section \ref{10.4b} allows us to describe infinitesimal deformations on the category of commutative $L\n$-algebras $(CAlg/L\n )$. Let $\varphi \colon \eL \rightarrow F(\N, A[[W\sb{1},W\sb{2}]])[[X]]$ be an infinitesimal deformation of the canonical morphism $\iota \colon \eL \rightarrow F(\N, A[[W\sb{1},W\sb{2}]])[[X]]$ for $A \in Ob\, (CAlg/L\n)$. 
Then there exist $e \in A$ and $b(W\sb{1}) \in A[[W\sb{1}]]$ such that $e-1$ and all the coefficients of the power series $b(W\sb{1})$ are nilpotent and such that
\begin{align*}
 \varphi(t  + W\sb{1} + X) &= e(t + W\sb{1}) + X,
  \\  
 \varphi (\sum\sb{n=1}^{\infty} X^{n}(-1)^{n+1}\frac{1}{[n]\sb{q}}(t+W\sb{1}))^{-n}Q^{-n})   &=\sum\sb{n=1}^{\infty} X^{n}(-1)^{n+1}\frac{1}{[n]\sb{q}}(e(t+W\sb{1}))^{-n}Q^{-n}) + b(W\sb{1}). 
\end{align*}
\begin{proposition}\label{6.8m}
We have an injection
\begin{multline*}
 \mathcal{F}\sb{L/k}(A) \rightarrow \hat{G}\sb{III} :=
  \{ (e,\,  b(W\sb{1})) \in A\times A[[W\sb 1]]\, | \, 
e-1 \text{\it \/  is nilpotent,} \\
\text{\it\/ all the coefficients 
 of $b(W\sb{1})$ are nilpotent} 
\} 
. 
\end{multline*}
\end{proposition}
\begin{conjecture}\label{6.8l}
If $q$ is not a root of unity , them
the injection in Proposition \ref{6.8m} is an equality. 
\end{conjecture}
$\hat{G}\sb{III}$ is a group functor on $(CAlg/L\n)$. In fact, for $A\in Ob(Alg/L\n)$, 
 we define the product of two elementss
$$(e, \, b(W\sb{1})),\,  (g, \, c(W\sb{1})) \in \hat{G}\sb{III}(A)$$ 
 by 
\[
 (e,\,  b(W\sb{1})) \star (g, \, c(W\sb{1})) := (eg, \, b(gW\sb{1} + (g-1)t)+ c(W\sb{1})). 
\]
Then the
product is a well-determined element of $\hat{G}\sb {III}(A)$, the product is associative,   
the 
 unit element is $(1,\, 0) \in \hat{G}\sb{III}
(A)$ and the inverse $(e, \, b(W\sb{1}))^{-1} = (e^{-1}, \, -b(e^{-1}W\sb{1}+ (e^{-1} -1)t))$. 
\par
So if Conjecture \ref{6.8l} is true, we have a non-splitting  exact sequence 
\[
 0 \rightarrow A[[W\sb{1}]]\sb{+} \rightarrow \infgal(L/k)(A) \rightarrow \hat{\G}\sb{m\, L\n}(A) \rightarrow 1,
\]
where $A[[W\sb{1}]]\sb{+}$ denote the additive group of the power series in $A[[W\sb{1}]]$ whose coefficients are nilpotent element. 
\subsubsection{Non-commutative deformations $\NCF$ for $(\com(t,\, \log\, t),\, \sigma,\,  \theta^{\ast})/\com$}
The arguments in Section \ref{10.4b} allows us to prove analogous results for \qsi field extension $(\com(t,\, \log\, t),\,  \sigma,\, \theta^{\ast})/\com$. We write assertions without giving detailsed proofs. For, the proofs are same. 
\begin{definition}
We introduce a functor 
\[
\widehat{QG}\sb{III \, q} \colon (\NCA) \rightarrow (Set)
\]
by setting 
\begin{multline*} 
\widehat{QG}\sb{III \, q}(A):=\{(H,\, \varphi(W\sb{1}))\in \gH\sb{q\, L\n}(A)\times A[[W\sb{1}]] \, | \, 
\text{\it\/ (1) $H = \begin{bmatrix}
e&f\\
0&1
\end{bmatrix} \in \widehat{\gH}\sb q (A)$ so that } \\
ef = qfe$, \, $e-1, f \in A \text{ are nilpotent. }
\text{\it\/(2) All the coefficients of $\varphi(W\sb{1})$}\\
\text{\it\/ are nilpotent. }
\text{\it\/(3) $[e(t + W\sb{1})+f, \, \varphi(W\sb{1})] = 0$.}
\}
\end{multline*}
\end{definition}
$\widehat{QG}\sb{III \, q}$ is a quantum formal group. Namely, for 
$$
(G,\,  \varphi(W\sb{1})),\, (H, \, \psi(W\sb{1})) \in \widehat{QG}\sb{III \, q}(A)
$$
 such that the two subsets 
\begin{align*}
&\text{\it 
\{ all the entries of matrix $H$, all the coefficients of the power series $\varphi(W\sb{1})$\},
} \\
&\text{\it
 \{all the entries of matrix $H$, all the coefficients of the power series $\psi(W\sb{1})$\}
 }
 \end{align*}
  of $A$ are mutually commutative, we define their product by 
\[
(G, \, \varphi(W\sb{1})) \star (H, \, \psi(W\sb{1})) := (GH, \, \varphi(gW\sb{1} + (g-1)t + h) + \psi(W\sb{2})), 
 \]
 where 
 $$
 H=\begin{bmatrix} 
 g & h \\ 
 0 & 1 
      \end{bmatrix}.
 $$
Then the product of two elements is a well-determined element in the set $\widehat{QG}\sb{III}(A)$ and the product is associative. The unit element is $(\Id \sb{2}, \,  0) \in \widehat{QG}\sb{III\, q}(A)$. The inverse 
$$
(G, \, \varphi (W\sb 1))^{-1}
= (G^{-1}, \, -
\varphi 
(
e^{-1}W\sb 1
+ (e^{-1} -1)t -e^{-1}f
)
 \in \widehat{QG}\sb{III\, q^{-1}}(A), 
$$
where
$$
G=\begin{bmatrix}
e & f \\
0  &1
\end{bmatrix}. 
$$
\begin{proposition}\label{6.8o}
We have a functorial injection
\[
\NCF(A) \rightarrow \widehat{QG}\sb{III \, q}(A)
\]
that send $\varphi \in \NCF(A)$ to $( \begin{bmatrix}
e & f \\
0& 1
\end{bmatrix}, \, b(W\sb{1})
)$. Here
\begin{align}
\varphi(t+ W\sb{1}+X)= e(t+W\sb{1}) + f + X,\\
\varphi (\sum\sb{n=1}^{\infty}X^{n} (-1)^{n+1} \frac{1}{[n]\sb{q}}(e(t+W\sb{1}) + f)^{-n}Q^{-n} + b(W\sb{1}) ).
\end{align}
\end{proposition}
We also have a Conjecture. 
\begin{conjecture}\label{6.8p}
If $q$ is not a root of unity, then 
the injection in Proposition \ref{6.8o} is an equality. So
\[
\NCF \simeq \widehat{QG}\sb{III}.
\]
Therefore quantum Galois group of the \qsi extension is the wuantum formal group $\widehat{QG}\sb{III\, q}$. 
\end{conjecture}
\begin{remark}
The argument in 
\ref{11.23a}
allows us to prove that Conjecture \ref{6.8p} implies Conjecture \ref{6.8j}.
\end{remark}
 \bibliographystyle{plain.bst}
\bibliography{umemura2b}

\begin{thebibliography}{10}

\bibitem{amaetal09}
Katsutoshi Amano, Akira Masuoka, and Mitsuhiro Takeuchi.
\newblock {\em Hopf algebra approach to {P}icard-{V}essiot thoery}, volume~6,
  pages 127--171.
\newblock Elsvier/{N}orth -{H}olland, Amsterdam, 2009.

\bibitem{and01}
Yves Andr\'e.
\newblock Diff\'erentielles non commutatives et th\'eorie de {G}alois
  diff\'erentielle ou aux diff\'erences. ({F}rench) [{N}oncommutative
  differentials and {G}alois theory for differential or difference equations].
\newblock {\em Ann. Sci. \'Ecole Norm. Sup. (4)}, 34:685--739, 2001.

\bibitem{har10}
Charlotte Hardouin.
\newblock {I}terative difference {G}alois theory.
\newblock {\em {J}ournal {R}eine {A}ngew. {M}ath.}, 644:101--144, 2010.

\bibitem{hay08}
Heidi Haynal.
\newblock {P}{I} degree parity in $q$-skew polynomial rings.
\newblock {\em J. {A}lgebra}, 319:4199--4221, 2008.

\bibitem{hei10}
Florian Heiderich.
\newblock {\em Galois {T}heory of {M}odule {F}ields}.
\newblock PhD thesis, Barcelona University, 2010.

\bibitem{mal01}
Bernard Malgrange.
\newblock Le groupo\"{\i}de de {G}alois d'un feuilletage. ({F}rench) [{T}he
  {G}alois groupoid of a foliation].
\newblock In {\em Essays on geometry and related topics, Vol. 1, 2}, Monogr.
  Enseign. Math., pages 465--501. Enseignement Math., Geneva, 2001.

\bibitem{man88}
Yu.~I. Manin.
\newblock {\em Quantum groups and noncommutative geometry}.
\newblock Centre de Recherches Math\'ematiques de l'universit\'e de Monr\'eal,
  Monr\'eal, 1988.

\bibitem{mori09}
Shuji Morikawa.
\newblock On a general difference {G}alois theory. {I}.
\newblock {\em Ann. Inst. Fourier (Grenoble)}, 59:2709--2732, 2009.

\bibitem{morume09}
Shuji Morikawa and Hiroshi Umemura.
\newblock On a general difference {G}alois theory. {II}.
\newblock {\em Ann. Inst. Fourier (Grenoble)}, 59:2733--2771, 2009.

\bibitem{swe69}
Moss Sweedler.
\newblock {\em Hopf {A}lgebras}.
\newblock Mathematics Lecture Series. Benjamin, New York, 1969.

\bibitem{ume96.2}
Hiroshi Umemura.
\newblock Differential {G}alois theory of infinite dimension.
\newblock {\em Nagoya Math. J.}, 144:59--135, 1996.

\bibitem{ume96.1}
Hiroshi Umemura.
\newblock Galois theory of algebraic and differential equations.
\newblock {\em Nagoya Math. J.}, 144:1--58, 1996.

\bibitem{ume06}
Hiroshi Umemura.
\newblock Galois theory and {P}ainlev\'e equations.
\newblock In {\em Th\'eories asymptotiques et \'equations de Painlev\'e},
  volume~14 of {\em S\'emin. Congr.}, pages 299--339. Soc. Math. France, Paris,
  2006.

\bibitem{ume07}
Hiroshi Umemura.
\newblock Invitation to {G}alois theory.
\newblock In {\em Differential equations and quantum groups}, volume~9 of {\em
  IRMA Lect. Math. Theor. Phys.}, pages 269--289. Eur. Math. Soc., Z\"urich,
  2007.

\bibitem{ume11}
Hiroshi Umemura.
\newblock {\em Picard-{V}essiot theory in general {G}alois theory}, volume~94,
  pages 263--293.
\newblock Banach Center Publ., 2011.

\end{thebibliography}
 \end{document}